\documentclass[11pt]{amsart}
\usepackage{latexsym,amssymb,amsmath,hyperref,amsthm,amsfonts,
caption,subcaption,tikz,comment}
\usepackage{mdwlist,dirtytalk,times}
\usepackage[capitalise, noabbrev]{cleveref}
\usetikzlibrary{
intersections, arrows.meta,
automata,er,calc,
backgrounds,
mindmap,folding,
patterns,
decorations.markings,
fit,decorations,
shapes,matrix,
positioning,
shapes.geometric,
arrows,through, graphs, graphs.standard
}
\usepackage{amsrefs}

\usepackage{blkarray}

\numberwithin{equation}{section}
\usetikzlibrary{positioning}
\usetikzlibrary{arrows}

\newtheorem{theorem}{Theorem}[section]
\newtheorem{lemma}[theorem]{Lemma}
\newtheorem{proposition}[theorem]{Proposition}
\newtheorem{corollary}[theorem]{Corollary}

\newtheorem{question}[theorem]{Question}

\theoremstyle{definition}
\newtheorem{definition}[theorem]{Definition} 
\newtheorem{remark}[theorem]{Remark}
\newtheorem{example}[theorem]{Example}
\newtheorem{acknowledgement}{Acknowledgement}

\newcommand{\K}{\mathbb{K}}

\newcommand{\N}{\mathbb{N}}

\newcommand{\R}{\mathbb{R}}
\newcommand{\Q}{\mathbb{Q}}
\newcommand{\Pp}{\mathbb{P}}

\newcommand{\tuple}[1]{\langle #1\rangle}
\newcommand{\st}{\colon}
\newcommand{\Ss}{\mathcal{S}}

\DeclareMathOperator{\co}{co}

\DeclareMathOperator{\HS}{HS}
\DeclareMathOperator{\HF}{HF}
\DeclareMathOperator{\lk}{link}
\DeclareMathOperator{\Soc}{Soc}
\DeclareMathOperator{\htt}{height}
\DeclareMathOperator{\sign}{sign}

\newcommand{\qand}{\quad \mbox{and} \quad}

\newcommand{\qfor}{\quad \mbox{for} \quad}

\newcommand{\qforevery}{\quad \mbox{for every} \quad}

\newcommand{\qwhere}{\quad \mbox{where} \quad}
\newcommand{\qif}{\quad \mbox{if} \quad}

\newcommand{\Sco}{\mathcal{S}^{\text{co}}}
\newcommand{\Kco}{\K^{\text{co}}}

\newcommand{\e}{\varepsilon}

\begin{document}
%%%%%%%%%%%%%%%%%%%%%%%%%%%%%%%%%%%%%%%%%%%%%%%%%%%%%%%
 
\title{Coinvariant stresses, Lefschetz properties and random complexes}

\author[T. Holleben]{Thiago Holleben}
\address[T. Holleben]
{Department of Mathematics \& Statistics,
Dalhousie University,
6297 Castine Way,
PO BOX 15000,
Halifax, NS,
Canada B3H 4R2}
\email{hollebenthiago@dal.ca}

\keywords{Stresses, Lefschetz properties, inverse systems, random complexes}
\subjclass[2020]{13E10, 13F55, 05E45, 05E40}

%%%%%%%%%%%%%%%%%%%%%%%%%%%%%%%%%%%%%%%%%%%%%%%%%%%%%%
 
\begin{abstract}
    Lefschetz properties and inverse systems have played key roles in understanding the $h$-vector of simplicial spheres. In 1996, Lee established connections between these two algebraic tools and rigidity theory, an area often used in the study of motions of geometric complexes. One of the key ideas, is to translate geometric information about a complex, coming from vertex coordinates, to the algebraic notion of a linear system of parameters. 
    In this paper, we explore similar connections in the nonlinear case, by using recent results of Herzog and Moradi (2021) where they prove that a subset of the elementary symmetric polynomials is always a system of parameters for the Stanley-Reisner ideal of a complex. 
    We investigate connections to the study of Lefschetz properties of monomial ideals. Using this perspective, we recover and extend the well known result of Migliore, Miró-Roig and Nagel on the failure of the WLP of monomial almost complete intersections, by showing that, with one simple exception, {\it every} homology sphere has a monomial artinian reduction failing the weak Lefschetz property.                                                      
     
    Finally, we state probabilistic consequences of our results under a model introduced by Linial and Meshulam. We prove that there exists an open interval for the probability parameter where failure of Lefschetz properties of monomial ideals should be expected.  
\end{abstract}

\maketitle

%%%%%%%%%%%%%%%%%%%%%%%%%%%%%%%%%%%%%%%%
%%%%%%%%%%%%%%%%%%%%%%%%%%%%%%%%%%%%%%%%%%%%%%%%

\section{Introduction}

One of the most important problems in the study of simplicial complexes is the $g$-conjecture proposed by McMullen~\cite{M1971}. The conjecture provides a characterization of the $h$-vectors of simplicial spheres, and was recently proven by Adiprasito~\cite{A2018}, with subsequent proofs by Petrotou-Vasiliki~\cite{PP2020}, and Adiprasito-Petrotou-Vasiliki~\cite{APP2021}.

Throughout the years, the conjecture was attacked from several different perspectives, including the theory of rigidity and stresses, and Stanley-Reisner theory. As a consequence, unexpected connections between these two approaches have been developed, starting with a paper by Lee~\cite{L1996}, where it is shown that inverse systems and Macaulay duality can be seen as a bridge between rigidity and Stanley-Reisner theory.

A common theme of the connections between the combinatorial and algebraic approaches to study simplicial spheres, is to translate geometric information such as coordinates of vertices into the algebraic notion of a linear system of parameters (abbreviated linear sop, or lsop). 
The main advantage of using a \emph{linear} sop, is that the algebraic information obtained by studying multiplication maps of linear forms translates directly to the conditions conjectured by McMullen. On the other hand, there does not exist a unique set of linear forms that can be used as a sop for every simplicial sphere.
In view of this drawback, finding alternative sets of homogeneous polynomials that can be used as a sop for every (or at least a large class of) simplicial spheres, becomes an interesting problem. 
 
It is well known in Algebraic Combinatorics that the quotient of the ideal $(e_1, \dots, e_n)$ generated by elementary symmetric polynomials defines a finite dimensional vector space. 
This standard fact implies, in particular, that a subset of elementary symmetric polynomials can be taken as a (non-linear) sop of arbitrary squarefree monomial ideals. This sequence of homogeneous elements is called the \emph{ universal system of parameters}, and has been used before for example in the work of De Concini, Eisenbud and Procesi~\cite{DCEP1982}, Garsia and Stanton~\cite{GS1984}, Smith~\cite{S1990}, Herzog and Moradi~\cite{HM2021} and Adams and Reiner~\cite{AR2023}. 

Motivated by Lee's definition of affine and linear stresses from~\cite{L1996}, and in view of the connections between the universal system of parameters above to the coinvariant ring of the symmetric group, we define \emph{coinvariant stresses} and \emph{$\ell$-coinvariant stresses}~(see~\cref{d:coinvariant}) of simplicial complexes. When $\Delta$ is a $d$-dimensional homology sphere, or equivalently a Cohen-Macaulay orientable pseudomanifold, it is shown that up to multiplication by scalars, $\Delta$ has a unique coinvariant stress of degree $\binom{d+2}{2}$. Our first main result is the following.

\begin{theorem}[{\bf The unique top coinvariant stress of a homology sphere}]\label{t:topintro}
    Let $\Delta$ be a $d$-dimensional homology sphere with facets $F_1, \dots, F_s$ and orientation $\e$. Set $x_{F_i} = \prod_{j \in F_i} x_j$. Then the unique top coinvariant stress of $\Delta$ is (up to multiplication by scalars):
    $$ 
        F_\Delta = \e(F_1) x_{F_1} V(F_1) + \dots + \e(F_s) x_{F_s} V(F_s),
    $$
    where $V(F_i) = \prod_{j < k, \{j,k\} \subset F_i} (x_j - x_k)$ is the Vandermonde determinant on variables $\{x_j \st j \in F_i\}$.
\end{theorem}
 
Applying~\cref{t:topintro} when $\Delta$ is the boundary of a simplex, we recover the classical fact that the Vandermonde determinant is the Macaulay dual generator of the ideal $(e_1,\dots,e_n)$ generated by elementary symmetric polynomials.

An artinian algebra $A$ is said to satisfy the \emph{weak Lefschetz property} (abbreviated WLP) if there exists a linear form $\ell \in A_1$ such that the multiplication maps $\times \ell: A_i \to A_{i + 1}$ have full rank for every $i$. If the maps $\times \ell^j: A_{i} \to A_{i + j}$ have full rank for every $i, j$, then $A$ satisfies the \emph{strong Lefschetz property} (abbreviated SLP). The study of Lefschetz properties began when Stanley~\cite{S80} showed that for every monomial complete intersection $I \subset R$, the artinian ring $R/I$ has the SLP.% It is not known whether for every complete intersection $I \subset R$, the artinian algebra $R/I$ has the SLP. 

One of the first interesting examples of a (level) monomial algebra failing the WLP was presented by Brenner and Kaid~\cite{BK2007}, where they used techniques from algebraic geometry to conclude that the algebra 
$$
    \frac{\K[x,y,z]}{(xyz, x^3, y^3, z^3)}
$$ 
fails the WLP. In~\cite{MMN2011}, Migliore, Miró-Roig and Nagel then used techniques from liaison theory to extend the example, showing that for every $n$, the algebra 
$$
    \frac{\K[x_1,\dots,x_n]}{(x_1\dots x_n, x_1^n,\dots, x_n^n)}
$$
fails the WLP. These algebras are called \emph{monomial almost complete intersections}, and can be defined as the quotient of $I_\Delta + (x_1^{d+2}, \dots, x_n^{d + 2})$, where $I_\Delta$ is the Stanley-Reisner ideal of the boundary of the $(d + 1)$--dimensional simplex. A key step in their proof is the fact that the Vandermonde determinant is in the kernel of a specific map.
  
When~\cref{t:topintro} is applied to the boundary of the simplex, we recover the Vandermonde determinant. A natural question then is whether the result about monomial almost complete intersections from~\cite{MMN2011} can be generalized, or at least proven using coinvariant stresses. Our second main result is the following result, which uses a slightly more general version of~\cref{t:topintro}.

\begin{theorem}[{\bf Nonvanishing homology and the failure of the WLP}]\label{t:failureintro}
    Let $d > 0$ and $\Delta$ be a $d$-dimensional complex such that $f_{d-1} \geq f_d$ and $\tilde H_d(\Delta; \K) \neq 0$. Denote the Stanley-Reisner ideal of $\Delta$ by $I_\Delta \subset R = \K[x_1,\dots, x_n]$. Then the algebra 
    $$
        \frac{R}{I_\Delta + (x_1^{d+2},\dots,x_n^{d+2})}
    $$
    fails the WLP.
\end{theorem}

\cref{t:failureintro} can be seen as a vast generalization of the result from~\cite{MMN2011}. More specifically, when $\Delta$ is the boundary of a simplex, we recover the fact that monomial almost complete intersections fail the WLP from~\cite{MMN2011}.

The techniques we use to prove~\cref{t:failureintro} are based on properties of $f$-vectors of simplicial complexes, inverse systems and liaison theory. In particular, the proof of~\cref{t:failureintro} depends on understanding inequalities that $f$-vector of special classes of simplicial complexes must satisfy. Although in the present paper we reduce the required inequalities to previously known ones, due to the wide range of techniques that have been used in the study of Lefschetz properties, it is clear that understanding Lefschetz properties of monomial ideals is a powerful tool for understanding the $f$-vectors of simplicial complexes.

\cref{t:failureintro,t:topintro} suggest the study of Lefschetz properties of monomial ideals is directly related to the study of (not necessarily linear) systems of parameters for monomial ideals. In the linear setting, several results on the positivity of $g$-numbers are known. Throughout the paper, we put some of these results in the context of Lefschetz properties of monomial ideals to show that certain multiplication maps are never surjective.

It should be noted that~\cref{t:failureintro} is a surprising result: in~\cite{NP2024}, Nagel and Petrović did extensive computations that suggest failure of the WLP is a rare phenomenon in most scenarios for level monomial algebras. \cref{t:failureintro} not only shows that a large class of ideals fails the WLP, but it shows that a property often associated with algebras having the WLP (being Gorenstein), in the monomial setting, can imply failure.

\begin{corollary}[{\bf The Gorenstein property in the monomial setting}]\label{c:gorensteinfailureintro}
    Let $I \subset R = \K[x_1, \dots, x_n]$ be a Gorenstein squarefree monomial ideal such that every variable of $R$ divides a generator of $I$, and $d + 1 = \dim R/I > 1$. Then the algebra 
    $$
        \frac{R}{I + (x_1^{d + 2}, \dots, x_n^{d + 2})}
    $$
    fails the WLP.
\end{corollary}

Note that although the assumption on the divisibility of generators by variables in~\cref{c:gorensteinfailureintro} is easily satisfied, the result is not true without it. When $I = 0$ the algebra has the SLP by a foundational result of Stanley~\cite{S80}.

The work of Nagel and Petrović in~\cite{NP2024} leads to several interesting questions, and most importantly, to the idea of {\emph{when should the WLP of monomial ideals be expected}}. \cref{t:failureintro} allows us to translate questions about how common are Lefschetz properties of monomial ideals, to questions about vanishing of homology groups of random complexes.
 
In~\cite{LM2006}, Linial and Meshulam considered a generalized Erdős–Rényi model, where the probability space $Y_d(n, p)$ consists of $d$-dimensional simplicial complexes $Y$ on $n$ vertices such that $Y$ contains every possible $(d-1)$--face, and the probability measure is given by 
$$
    \Pp(Y \in Y_d(n, p)) = p^{f_d} (1-p)^{\binom{n}{d + 1} - f_d},
$$
where $f_d$ is the number of $d$-faces of $Y$.

Due to the original results from~\cite{ER1960} on the classical Erdős–Rényi model, a natural question to consider is if there exists a threshold for the value of $p = \frac{c}{n}$ such that 
$$
    \lim_{n \to \infty} \Pp(Y \in Y_d(n, \textstyle \frac{c}{n}) \st \tilde H_d(Y; \K) \neq 0) = 1.
$$
This question and its variations have been approached by several authors (see for example~\cites{K1987,ALLTM2013,MW2009}). We apply~\cref{t:failureintro} together with one of the main results of~\cite{ALLTM2013} to show the following.

\begin{theorem}[{\bf Sometimes failure should be expected}]
    Given a $d$-dimensional simplicial complex $Y \in Y_d(n, p)$, let $A_Y = \frac{\K[x_1,\dots, x_n]}{I_Y + (x_1^{d + 2}, \dots, x_n^{d + 2})}$, where $I_Y$ is the Stanley-Reisner ideal of $Y$ and $\K$ is a field.
    \begin{enumerate}
        \item If $d = \lceil sn \rceil < n$ for some $\frac{1}{2} \leq s < 1$, then for every $p$:
        $$
            \Pp(Y \in Y_d(n, p) \st A_Y \mbox{ fails the WLP}) \geq \Pp(Y \in Y_d(n, p) \st \tilde H_d(Y; \K) \neq 0).
        $$
        \item For every $d > 0$, there exists a non empty interval $(c_d, d + 1)$ such that for every $c \in (c_d, d + 1)$:
        $$
            \lim_{n \to \infty} \Pp(Y \in Y_d(n, \textstyle \frac{c}{n}) \st A_Y \mbox{ fails the WLP}) = 1.
        $$
    \end{enumerate}
\end{theorem}
Finally, in ~\cref{s:future} we collect questions that follow from our work.
% In this text, we focus on the connections between stresses and inverse systems, defining the notion of {\textit{coinvariant stresses}} of a $d$-dimensional Gorenstein* complex $\Delta$, and giving a formula for the unique coinvariant stress of degree $(d+1)(d+2)/2$ of $\Delta$. We then focus on connections between our work and different topics from Combinatorics and Commutative Algebra, stating a coinvariant analogue of the $g$-conjecture, and recovering previous results of Migliore, Miró-Roig and Nagel~\cite{MMN2011} on the failure of Lefschetz properties of monomial almost complete intersections.

\section{Preliminaries}

A {\bf simplicial complex} $\Delta$ is a collection of $V$, called {\bf faces}, such that if $\sigma \in \Delta$ and $\tau \subset \sigma$, then $\tau \in \Delta$. Maximal faces of $\Delta$ are called {\bf facets}, if the facets of $\Delta$ are $F_1, \dots, F_s$, we write $\Delta = \tuple{F_1, \dots, F_s}$. If every facet of $\Delta$ has the same size, we say $\Delta$ is {\bf pure}. The dimension of a face $\sigma$ of $\Delta$ is $\dim \sigma = |\sigma| - 1$, and the dimension of $\Delta$ is $\dim \Delta = \max (\dim \sigma \st \sigma \in \Delta)$. If $\Delta$ is a $d$--dimensional simplicial complex, its $(d-1)$--dimensional faces are also called {\bf ridges}. Given a face $\sigma \in \Delta$, the {\bf link} of $\sigma$ in $\Delta$, is the simplicial complex $\lk_\Delta(\sigma) = \{\tau \st \tau \cup \sigma \in \Delta \qand \tau \cap \sigma = \emptyset\}$. A simplicial complex $\Delta$ is said to be a {\bf cone} over one of its faces $\sigma$ if $\sigma \cup \tau \in \Delta$ for every $\tau \in \Delta$. To every simplicial complex $\Delta$, we may associate its geometric realization $|\Delta|$, a topological space that is a subset of $\R^k$ for some $k$, and to every face $\sigma$ of $\Delta$, $|\Delta|$ contains the convex hull of the points corresponding to vertices of $\sigma$ (see~\cite[Chapter 1]{M2003} for more details on the construction). If $|\Delta|$ is homeomorphic to a sphere, we say $\Delta$ is a {\bf simplicial sphere}.

The {\bf $f$-vector} of a $d$-dimensional complex $\Delta$ is the sequence of numbers $f(\Delta) = (f_{-1}, \dots, f_d)$, where $f_i$ is the number of $i$-dimensional faces of $\Delta$. The {\bf $h$-vector} of $\Delta$ is the sequence of numbers $h(\Delta) = (h_0, \dots, h_{d + 1})$ where the entries are given by the polynomial relations 
$$
    \sum_{i = 0}^{d + 1} f_{i - 1} (t - 1)^{d + 1 - i} = \sum_{i = 0}^{d + 1} h_i t^{d + 1 - i}.
$$
The {\bf $g$-vector} of $\Delta$ is the sequence of numbers $g(\Delta) = (g_0, \dots, g_{\lceil \frac{d + 1}{2} \rceil})$, where $g_0 = 1$ and $g_i = h_i - h_{i - 1}$ for $i > 0$.

\begin{definition}
    A simplicial complex $\Delta$ is a $d$-dimensional {\bf pseudomanifold} if the following conditions hold:
    \begin{enumerate}
        \item $\Delta$ is pure,
        \item $\Delta$ is {\bf strongly connected}, in other words, for every pair of facets $A, B$ of $\Delta$, there exists a sequence of facets $G_1, \dots, G_r$ such that $G_i \cap G_{i+1}$ is a $(d-1)$--dimensional face for every $i$, and $A = G_1$, $B = G_r$, and
        \item every ridge of $\Delta$ is contained in at most $2$ facets of $\Delta$.
    \end{enumerate}
    The {\bf boundary} of a pseudomanifold $\Delta$ is the set of ridges of $\Delta$ contained in only one facet. If the boundary of $\Delta$ is empty, we say $\Delta$ is a pseudomanifold without boundary.
  
    A $d$-dimensional pseudomanifold without boundary $\Delta = \tuple{F_1, \dots, F_s}$ is said to be {\bf orientable} over a field $\K$ if $\tilde H_{d}(\Delta; \K) \neq 0$. A nonzero element $\varepsilon = \varepsilon(F_1) F_1 + \dots + \varepsilon(F_s) F_s \in \tilde H_d(\Delta; \K)$ is called an {\bf orientation} of $\Delta$.

    %    Moreover, we say $\Delta = \tuple{F_1, \dots, F_s}$ is an {\bf orientable pseudomanifold} if it is a pseudomanifold without boundary and there exists a map $\e: \{F_1, \dots, F_s\} \to \{-1, 1\}$ such that \textcolor{red}{find nice definition}. The map $\e$ is called an {\bf orientation} of $\Delta$.
\end{definition}

% An ordered face of a simplicial complex $\Delta$ is a list $(v_1, \dots, v_{k + 1})$ where $\{v_1, \dots, v_{k + 1}\}$ is a face of $\Delta$. Let $\Ord_k(\Delta)$ denote the set of ordered $k$-faces of $\Delta$.
% %\cite{GHphd}
% \begin{definition}[{\bf Orientable pseudomanifolds}]\label{d:orientable-pseudomanifold}
%     Let $\Delta$ be an $d$-dimensional simplicial pseudomanifold without boundary. We say $\Delta$ is {\bf orientable} over a field $\K$ of characteristic zero if there exists a map $\e: \Ord_d(\Delta) \to \Z$ such that 
%     \begin{enumerate}
%         \item For every facet $F = \{v_1, \dots, v_{n+1}\}$ we have 
%         $$
%             \e(v_1, \dots, v_{d+1}) = \pm 1
%         $$
%         \item For every permutation $\pi$ of $[d + 1] = \{1, \dots, d + 1\}$ and every facet $\{v_1, \dots, v_{d+1}\}$ we have 
%         $$
%             \e(v_1, \dots, v_{d+1}) = \text{sign}(\pi)\e(v_{\pi(1)}, \dots, v_{\pi(d+1)})
%         $$
%         where sign$(\pi)$ is $-1$ for odd and $1$ for even permutations
%         \item For every ridge $\sigma = \{v_1, \dots, v_d\}$ and two facets $\{v_1, \dots, v_{d}, v'_{d+1}\}$ and $\{v_1, \dots, v_{d}, v''_{d+1}\}$ containing $\sigma$, we have 
%         $$
%             \e(v_1, \dots, v_d, v'_{d+1}) = -\e(v_1, \dots, v_d, v''_{d+1})
%         $$
%     \end{enumerate}
% \end{definition}

\begin{example}[{\bf The projective plane: a non-orientable pseudomanifold}]\label{e:projective}
    Let 
    $$
        \Sigma = \tuple{123,125,245,345,134,146,156,356,236,246}
    $$
    be the following triangulation of $\R \Pp^2$.
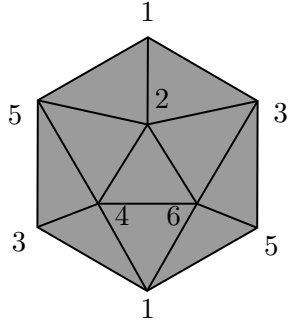
    \begin{center}

\tikzset{every picture/.style={line width=0.75pt}} %set default line width to 0.75pt        

\begin{tikzpicture}[x=0.75pt,y=0.75pt,yscale=-1,xscale=1]

%Shape: Regular Polygon [id:dp9867867516463287] 
\draw  [fill={rgb, 255:red, 155; green, 155; blue, 155 }  ,fill opacity=1 ] (184.79,159) -- (129.47,126.82) -- (129.68,62.82) -- (185.21,31) -- (240.53,63.18) -- (240.32,127.18) -- cycle ;
%Straight Lines [id:da8367588863173534] 
\draw    (185,75) -- (185.21,31) ;
%Shape: Triangle [id:dp021308096260104437] 
\draw   (185,75) -- (209.92,115) -- (160.08,115) -- cycle ;
%Straight Lines [id:da4971660395817197] 
\draw    (185,75) -- (129.68,62.82) ;
%Straight Lines [id:da031009124592503268] 
\draw    (185,75) -- (240.53,63.18) ;
%Straight Lines [id:da40269647271523956] 
\draw    (240.32,127.18) -- (209.92,115) ;
%Straight Lines [id:da7262535397349787] 
\draw    (129.47,126.82) -- (160.08,115) ;
%Straight Lines [id:da33905689223742597] 
\draw    (184.79,159) -- (160.08,115) ;
%Straight Lines [id:da7851811880914239] 
\draw    (184.79,159) -- (209.92,115) ;
%Straight Lines [id:da32736136025626705] 
\draw    (160.08,115) -- (129.68,62.82) ;
%Straight Lines [id:da22666474674424797] 
\draw    (209.92,115) -- (240.53,63.18) ;

% Text Node
\draw (180,12) node [anchor=north west][inner sep=0.75pt]   [align=left] {$\displaystyle 1$};
% Text Node
\draw (247,63) node [anchor=north west][inner sep=0.75pt]   [align=left] {$\displaystyle 3$};
% Text Node
\draw (113,64) node [anchor=north west][inner sep=0.75pt]   [align=left] {$\displaystyle 5$};
% Text Node
\draw (187.11,56) node [anchor=north west][inner sep=0.75pt]   [align=left] {$\displaystyle 2$};
% Text Node
\draw (167,115) node [anchor=north west][inner sep=0.75pt]   [align=left] {$\displaystyle 4$};
% Text Node
\draw (193,115) node [anchor=north west][inner sep=0.75pt]   [align=left] {$\displaystyle 6$};
% Text Node
\draw (115,128) node [anchor=north west][inner sep=0.75pt]   [align=left] {$\displaystyle 3$};
% Text Node
\draw (242.32,130.18) node [anchor=north west][inner sep=0.75pt]   [align=left] {$\displaystyle 5$};
% Text Node
\draw (180,162) node [anchor=north west][inner sep=0.75pt]   [align=left] {$\displaystyle 1$};

\end{tikzpicture}
    \end{center}
    The complex $\Sigma$ is an example of a non-orientable pseudomanifold over a field $\K$ of characteristic zero.
\end{example}

% \begin{remark}
%     There are other definitions for orientable pseudomanifolds that are often preferred. In later sections we will need~\cref{d:orientable-pseudomanifold} specifically because of the explicit definition of orientation.
% \end{remark}

A simplicial complex $\Delta$ is called a {\bf $\K$-homology sphere} if 
$$
    \tilde H_i(\lk_\Delta(\sigma); \K) \cong \begin{cases}
        \K \qif i = \dim \lk_\Delta(\sigma) \\
        0 \quad \mbox{ otherwise}
    \end{cases} \qforevery \sigma \in \Delta.
$$
Given $\tau \subset V$ and $R = \K[x_v \st v \in V]$, we set $x_\tau = \prod_{i \in \tau} x_i$. The {\bf Stanley-Reisner ideal} of a simplicial complex $\Delta$ is the squarefree monomial ideal $I_\Delta = (x_\tau \st \tau \not \in \Delta) \subset \K[x_v \st v \in V]$. 
For a homogeneous ideal $I \subset R$, we denote the {\bf Hilbert series} of $R/I$, and the {\bf Hilbert function} of $R/I$ by
$$
    \HS_{\frac{R}{I}}(t) = \sum_i \dim \Big{(}\frac{R}{I}\Big{)}_i t^i \qand \HF_{\frac{R}{I}}(n) = \dim \Big{(}\frac{R}{I}\Big{)}_n. 
$$

The following result due to Stanley~\cite{S1996B} will be useful in later sections.

\begin{theorem}[{\cite[Theorem 5.1, Chapter 2]{S1996B}}]\label{t:gorensteinequivalences}
    Let $\Delta$ be a simplicial complex on vertex set $V$ such that $\Delta$ is not a cone over any of its vertices, $\K$ a field and $I_\Delta \subset R = \K[x_v \st v \in V]$ its Stanley-Reisner ideal. The following are equivalent
    \begin{enumerate}
        \item $R/I_\Delta$ is Gorenstein,
        \item $\Delta$ is a $\K$-homology sphere
        \item $\Delta$ is a Cohen-Macaulay orientable pseudomanifold over $\K$.
    \end{enumerate}
\end{theorem}

\begin{example}[{\bf The pinched torus: A non-CM orientable pseudomanifold}]\label{e:pinched}
    Let 
    $$ 
        \Lambda  = \tuple{123, 134, 145, 125, 239, 289, 257, 278, 457, 674, 346, 369, 781, 671, 691, 891}
    $$
    be the simplicial complex obtained by identiying vertices $0$ and $1$ in the simplicial complex $\Gamma$ below.

    \begin{center}

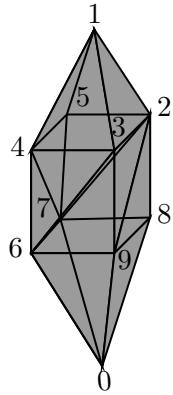
\tikzset{every picture/.style={line width=0.75pt}} %set default line width to 0.75pt        

\tikzset{every picture/.style={line width=0.75pt}} %set default line width to 0.75pt        

\begin{tikzpicture}[x=0.75pt,y=0.75pt,yscale=-1,xscale=1]

%Shape: Cube [id:dp2635109284032633] 
\draw  [fill={rgb, 255:red, 255; green, 255; blue, 255 }  ,fill opacity=1 ] (100,139) -- (118,121) -- (160,121) -- (160,173) -- (142,191) -- (100,191) -- cycle ; \draw   (160,121) -- (142,139) -- (100,139) ; \draw   (142,139) -- (142,191) ;
%Straight Lines [id:da1579306513340082] 
\draw [fill={rgb, 255:red, 128; green, 128; blue, 128 }  ,fill opacity=1 ]   (142,191) -- (142,139) ;
%Shape: Polygon [id:ds2923394536455506] 
\draw  [fill={rgb, 255:red, 155; green, 155; blue, 155 }  ,fill opacity=1 ] (100,191) -- (100,139) -- (142,139) -- cycle ;
%Shape: Polygon [id:ds5165327908581541] 
\draw  [fill={rgb, 255:red, 155; green, 155; blue, 155 }  ,fill opacity=1 ] (100,191) -- (142,139) -- (142,191) -- cycle ;
%Shape: Polygon [id:ds15645543762394576] 
\draw  [fill={rgb, 255:red, 155; green, 155; blue, 155 }  ,fill opacity=1 ] (142,191) -- (142,139) -- (160,121) -- cycle ;
%Shape: Polygon [id:ds547065482285193] 
\draw  [fill={rgb, 255:red, 155; green, 155; blue, 155 }  ,fill opacity=1 ] (160,173) -- (142,191) -- (160,121) -- cycle ;
%Shape: Polygon [id:ds749086672071618] 
\draw  [fill={rgb, 255:red, 155; green, 155; blue, 155 }  ,fill opacity=1 ] (100,139) -- (131.83,78) -- (118,121) -- cycle ;
%Shape: Polygon [id:ds4635776632217494] 
\draw  [fill={rgb, 255:red, 155; green, 155; blue, 155 }  ,fill opacity=1 ] (142,139) -- (100,139) -- (131.83,78) -- cycle ;
%Shape: Polygon [id:ds7406268383556209] 
\draw  [fill={rgb, 255:red, 155; green, 155; blue, 155 }  ,fill opacity=1 ] (142,139) -- (131.83,78) -- (160,121) -- cycle ;
%Shape: Polygon [id:ds917484603239499] 
\draw  [fill={rgb, 255:red, 155; green, 155; blue, 155 }  ,fill opacity=1 ] (142,191) -- (100,191) -- (135.83,248) -- cycle ;
%Shape: Polygon [id:ds15334174676703216] 
\draw  [fill={rgb, 255:red, 155; green, 155; blue, 155 }  ,fill opacity=1 ] (142,191) -- (135.83,248) -- (160,173) -- cycle ;
%Straight Lines [id:da13577205271971193] 
\draw    (100,139) -- (118,121) ;
%Straight Lines [id:da9878770575524991] 
\draw    (160,121) -- (118,121) ;
%Straight Lines [id:da9445015356401882] 
\draw    (118,121) -- (131.83,78) ;
%Straight Lines [id:da14112760007123515] 
\draw    (100,191) -- (114.83,174) ;
%Straight Lines [id:da8206819824758897] 
\draw    (114.83,174) -- (118,121) ;
%Straight Lines [id:da5387437833698432] 
\draw    (114.83,174) -- (160,173) ;
%Straight Lines [id:da6341839167342591] 
\draw    (114.83,174) -- (100,139) ;
%Straight Lines [id:da8186624865921379] 
\draw    (114.83,174) -- (160,121) ;
%Shape: Boxed Polygon [id:dp8810981385658234] 
\draw   (135.83,248) -- (114.83,174) -- (100,191) -- cycle ;

% Text Node
\draw (127,64) node [anchor=north west][inner sep=0.75pt]   [align=left] {$\displaystyle 1$};
% Text Node
\draw (162,111) node [anchor=north west][inner sep=0.75pt]   [align=left] {$\displaystyle 2$};
% Text Node
\draw (139,121) node [anchor=north west][inner sep=0.75pt]   [align=left] {$\displaystyle 3$};
% Text Node
\draw (88,131) node [anchor=north west][inner sep=0.75pt]   [align=left] {$\displaystyle 4$};
% Text Node
\draw (121,106) node [anchor=north west][inner sep=0.75pt]   [align=left] {$\displaystyle 5$};
% Text Node
\draw (131.38,249) node [anchor=north west][inner sep=0.75pt]  [xslant=-0.07] [align=left] {$\displaystyle 0$};
% Text Node
\draw (87,182) node [anchor=north west][inner sep=0.75pt]   [align=left] {$\displaystyle 6$};
% Text Node
\draw (101,162) node [anchor=north west][inner sep=0.75pt]   [align=left] {$\displaystyle 7$};
% Text Node
\draw (162,165) node [anchor=north west][inner sep=0.75pt]   [align=left] {$\displaystyle 8$};
% Text Node
\draw (142,188) node [anchor=north west][inner sep=0.75pt]   [align=left] {$\displaystyle 9$};

\end{tikzpicture}
    \end{center}
    The complex $\Lambda$ is a triangulation of the {\bf pinched torus}, it is an example of a non-Cohen-Macaulay orientable pseudomanifold. The complex $\Gamma$ above is an example of a homology sphere. The complex $\Gamma$ is in fact a simplicial sphere.
\end{example}

\section{Inverse systems and (coinvariant) stresses}

In 1971, McMullen~\cite{M1971} asked whether the set of $h$-vectors of a class of complexes called simplicial polytopes is equal to set of vectors $h = (h_0, \dots, h_{d})$ satisfying:
\begin{enumerate}
    \item $h_0 = 1$;
    \item $h_i = h_{d - i}$; and
    \item there exists a monomial ideal $I \subset R$ such that $\HF_{\frac{R}{I}}(i) = h_i - h_{i - 1}$ for $i = 1, \dots, \lceil \frac{d}{2} \rceil$ and $\HF_{\frac{R}{I}}(i) = 0$ for $i > \lceil \frac{d}{2} \rceil$.
\end{enumerate}
This question is now called the $g$-conjecture. Billera and Lee~\cite{BL1981} showed that every sequence satisfying the properties above is the $h$-vector of some simplicial polytope, while Stanley~\cite{S80.1} proved that the $h$-vector of every simplicial polytope satisfies the properties above. Stanley's proof was then generalized by Adiprasito~\cite{A2018} (see also~\cites{PP2020,APP2021}) to the class of simplicial spheres.

A weaker question then the one asked by McMullen, is whether the entries of the $g$-vector of a simplicial polytope are nonnegative. In 1987, Kalai~\cite{K1987} gave a new proof that $g_2(P)$ is nonnegative for any simplicial polytope $P$. Kalai's proof used techniques from \textit{rigidity theory}, which led to the natural question of whether it is possible to generalize the classical notions of stresses and infinitesimal motions from rigidity theory to higher dimensions, in order to strengthen the connections between $g$-vectors and this theory.

In 1996, Lee~\cite{L1996} answered this question in the positive, defining the notions of affine and linear stresses, and connecting rigidity theory to Stanley's proof of the $g$-conjecture mentioned above. The key algebraic notion that ties these ideas together was developed by Macaulay, which we now define.

\begin{definition}[{\bf Inverse systems}] 
    Let $\K$ be a field, $R = \K[x_1,\dots, x_n]$, $S = \K[y_1, \dots, y_n]$, where $\deg x_i = 1$ and $\deg y_i = -1$. Let $x_a = x_1^{a_1} \dots x_n^{a_n} \in R$ and $y_b = y_1^{b_1} \dots y_n^{b_n} \in S$. Then we can view $S$ as an $R$-module via two different actions, which we define on monomials, but are easily extended (linearly) to arbitrary elements.
    \begin{enumerate}
        \item {\bf (Differentiation)} $x_a \bullet y_b = \frac{\partial y_b}{\partial y_1^{a_1}\dots \partial y_n^{a_n}}$, and 
        \item {\bf (Contraction)} $x_a \circ y_b = \begin{cases}
            y_1^{b_1 - a_1} \dots y_n^{b_n - a_n} \quad \mbox{ if $b_i \geq a_i$ for every $i$.} \\
            0 \quad \mbox{otherwise.}
        \end{cases}$
    \end{enumerate}
    Let $I \subset R$ be an ideal. The {\bf inverse system} of $I$ is the $R$-module (depending on the $R$-action) 
    \begin{enumerate}
        \item $I^\perp = \{F \in S \st g \bullet F = 0 \qforevery g \in I\} \subset S$.
        \item $I^{-1} = \{F \in S \st g \circ F = 0 \qforevery g \in I\} \subset S$.
    \end{enumerate}
\end{definition}

Throughout this paper, we will only consider $S$ an $R$-module with the action given by contraction. Since Lee's work uses differentiation, we include the definition here for completeness.

Given a homogeneous ideal $I \subset R$ such that $\dim \frac{R}{I} = d$, a {\bf system of parameters} (abbreviated sop) is a sequence of homogeneous elements $f_1, \dots, f_d$ where $\frac{R}{I + (f_1, \dots, f_d)}$ is a finite dimensional $\K$-vector space, or in other words, $I + (f_1, \dots, f_d)$ is an artinian ideal. If moreover $f_i$ is a linear form for every $i$, we say that $f_1, \dots, f_d$ is a {\bf linear system system of parameters} (abbreviated lsop).

The following result due to Macaulay~\cite{macaulaybookinversesystems} (see~\cite[Chapter 2]{lefschetzbook} for a modern treatment of this result) is the key algebraic result we need.
\begin{theorem}[{\bf Macaulay inverse system duality,~\cite{macaulaybookinversesystems}}]\label{t:macaulayduality}
    Let $I$ be an artinian ideal. Then $I^{-1}$ is a finite cyclic module if and only if $R/I$ is an Artinian Gorenstein ring. In this case, the top nonzero degree of $A = R/I$ is equal to $\deg F$, where $F$ is a generator of $I^{-1}$.

    Moreover for any homogeneous artinian ideal $I$, $\dim (R/I)_i = \dim I^{-1}_{-i}$.
\end{theorem}

\begin{remark}
    Since in the context of inverse systems the vector spaces $R$ and $S$ are isomorphic, we will often think of elements in $S$, as elements in $R$.
\end{remark}
 
Given a $d$-dimensional Cohen-Macaulay simplicial complex $\Delta$ such that $V(\Delta) = [n]$, and a function $p: [n] \to \R^{d + 1}$, Lee's idea to bridge Kalai's~\cite{K1987} and Stanley's~\cite{S80.1} approaches to the study of $g$-vectors comes from translating the data from $p$, to a lsop of the Stanley-Reisner ideal $I_\Delta$. More specifically, assuming the function $p$ is generic (i.e. the multiset of coordinates of $p(i)$, for $i \in V(\Delta)$ is algebraically independent over $\Q$), the set of linear forms 
\begin{equation}\label{eq:plinear}
    \theta_j = \sum_i p(i)_j x_i \in R \qfor j = 1, \dots, d + 1
\end{equation}
is a lsop of $I_\Delta$, where $p(i)_j$ denotes the $j$-th coordinate of $p(i)$. We are now ready to state the definition of affine and linear stresses.
\begin{definition}[{\bf Affine and linear stresses},~\cite{L1996}]
    Following notation as above, the space of {\bf $k$-linear stresses} of the pair $(\Delta, p)$ is the vector space
    $$
        \Ss^l_k(\Delta, p) = (I_\Delta,\theta_1,\dots,\theta_{d+1})^{\perp}_{-k}.
    $$
    Moreover, the space of {\bf $k$-affine stresses} of the pair $(\Delta, p)$ is the vector space 
    $$
        \Ss^a_k(\Delta, p) = (I_\Delta, \theta_1, \dots, \theta_{d+1}, L)^{\perp}_{-k},
    $$
    where $L = x_1 + \dots + x_n \in R$.
\end{definition}

Lee showed that the coefficients of polynomials in $\Ss^a_2(\Delta, p)$ are directly related to the standard notion of stresses from bar-and-joint structures from rigidity theory. To see why inverse systems are intrinsically related to the study of $h$-vectors of (Cohen-Macaulay) simplicial complexes, we need the following lemma.

\begin{lemma}\label{l:basicHS}
    Let $\Delta$ be a $d$-dimensional Cohen-Macaulay complex, and $f_1, \dots, f_{d+1}$ a sop of $I_\Delta$. Then 
    $$
        \HS_{\frac{R}{(I_\Delta, f_1, \dots, f_{d+1})}}(t) = h_\Delta(t)\prod_{i = 1}^{d + 1}(1 + t + \dots + t^{\deg f_i - 1}).
    $$
\end{lemma}

\begin{proof}
    By~\cite[Theorem 2.1.2]{CMrings}, since $\Delta$ is Cohen-Macaulay and $f_1, \dots, f_{d+1}$ is a sop of $I_\Delta$, this sequence of elements is also a regular sequence of $I_\Delta$. Hence for every $i$ we have the following short exact sequence.
    $$
        0 \to \frac{R}{I_\Delta + (f_1, \dots, f_i)}(-\deg f_{i + 1}) \xrightarrow{\times f_{i+1}} \frac{R}{I_\Delta + (f_1, \dots, f_i)} \rightarrow \frac{R}{I_\Delta + (f_1, \dots, f_{i+1})} \to 0
    $$
    We then conclude 
    $$
        \HS_{\frac{R}{(I_\Delta, f_1,\dots, f_{i+1})}}(t) = (1 - t^{\deg f_{i + 1}})\HS_{\frac{R}{(I_\Delta,f_1,\dots,f_i)}}(t)
    $$
    and finally by induction:
    $$
    \HS_{\frac{R}{(I_\Delta, f_1, \dots, f_{d + 1})}} = \HS_{\frac{R}{I_\Delta}}(t) \prod_{i = 1}^{d + 1}(1 - t^{\deg f_i}) = h_\Delta(t) \prod_{i = 1}^{d + 1}\frac{(1 - t^{\deg f_i})}{(1 - t)} = h_\Delta(t) \prod_i^{d + 1}(1 + t + \dots + t^{\deg f_i - 1}).
    $$
\end{proof}

In the special case where $f_1, \dots, f_{d+1}$ is a lsop,~\cref{l:basicHS,t:macaulayduality} imply 
    \begin{equation}\label{eq:linearstress}
        \dim \Ss^l_k(\Delta, p) = h_i.        
    \end{equation}
Equation \eqref{eq:linearstress} shows that the use of lsops is a useful tool for the study of $h$-vectors. A drawback of using linear forms, is that the set of linear forms changes depending on the complex $\Delta$.

Our alternative approach to avoid this drawback, is to use a \emph{non linear} sop that works for any simplicial complex. Note that the trade-off is that information on the Hilbert series will not translate directly to the study of $h$-vectors.

The sop that we use comes from algebraic combinatorics. Given a fixed integer $n$, we denote by $e_i$ the {\bf $i$-th elementary symmetric polynomial}
$$
    e_i = \sum_{j_1 < \dots < j_i} x_{j_1} \dots x_{j_i} \qfor i \leq n.
$$
It is a well known fact that the ideal $I = (e_1, \dots, e_n) \subset \K[x_1,\dots, x_n]$ is a complete intersection, and the quotient $\frac{R}{(e_1, \dots, e_n)}$ is called the {\bf coinvariant ring} of the symmetric group $S^n$ on $n$ elements. It is also known (see~\cites{HM2021,DCEP1982,GS1984,S1990,AR2023}) that the set of polynomials $e_1, \dots, e_{d+1}$ is a sop for any $d$-dimensional simplicial complex. This sop is called the {\bf universal system of parameters}. Here we give a proof of this fact in a slightly more general setting. This simple observation suggests how to find families of polynomials with this property.

\begin{lemma}\label{l:ci}
    Let $I = (f_1, \dots, f_n) \subset R = \K[x_1, \dots, x_n]$ be a complete intersection such that for every $d$-dimensional complex $\Delta$, $f_i \in I_\Delta$ for every $i > d + 1$. Then $f_1, \dots, f_{d+1}$ is a sop for every $I_\Delta$ where $\Delta$ is a $d$-dimensional simplicial complex.
\end{lemma}

\begin{proof}
    Since $I$ is a complete intersection on $n = \dim R$ generators, we know $\frac{R}{I}$ is a finite dimensional vector space. By assumption we know $I \subseteq I_\Delta + (f_1, \dots, f_{d+1})$, hence $\frac{R}{I_\Delta + (f_1, \dots, f_{d+1})}$ is also a finite dimensional vector space for every $d$-dimensional complex $\Delta$.
\end{proof}

% \begin{remark}
%     \cref{l:ci} together with the standard fact that the set of complete symmetric functions defines the same ideal as $(e_1, \dots, e_n)$, shows that for any monomial ideal $I \subset R = \K[x_1, \dots, x_n]$ such that $\dim \frac{R}{I} = d + 1$, the set of polynomials $f_1, \dots, f_{d + 1}$, where $f_i$ is the complete symmetric function, is a sop of $I$.
% \end{remark}

We may now define the notion of a coinvariant stress of a simplicial complex.

\begin{definition}[{\bf Coinvariant and $\ell$-coinvariant stresses}]\label{d:coinvariant}
    Let $\Delta$ be a $d$-dimensional simplicial complex. The space of {\bf $k$-coinvariant stresses} of $\Delta$ is the vector space 
    $$
        \dim \Sco_k(\Delta) = (I_\Delta, e_1, \dots, e_{d+1})^{-1}_{-k}.
    $$
    Moreover, given a linear form $\ell \in R$, the space of {\bf $k$-$\ell$-coinvariant stresses} is the vector space 
    $$
        \dim \Sco_k(\Delta, \ell) = (I_\Delta, e_1, \dots, e_{d + 1}, \ell)^{-1}_{-k}.
    $$
    Lastly, we set
    $$
        \Kco(\Delta) = \frac{R}{I_\Delta + (e_1, \dots, e_{d + 1})}.
    $$
\end{definition}

Note that for any $d$-dimensional Cohen-Macaulay complex $\Delta$,~\cref{l:basicHS} directly implies the following.

\begin{corollary}
    Let $\Delta$ be a $d$-dimensional Cohen-Macaulay complex. Then 
    $$
        \HS_{\Kco(\Delta)}(q) = h_\Delta(q)[d]_q!
    $$
    where $[d]_q! = \prod_{i = 1}^{d}(1 + q + \dots + q^{i})$ is the $q$-analogue of $d!$.

    In particular, if $h_{d+1} > 0$, the top nonzero degree of $\Kco(\Delta)$ is $\binom{d + 2}{2}$.
\end{corollary}

\begin{proof}
    We only need to prove the last statement, which follows directly since the Hilbert series of $\Kco(\Delta)$ is a polynomial of degree $d + 1 + \binom{d + 1}{2} = \binom{d + 2}{2}$.
\end{proof}

\section{The (coinvariant) Macaulay dual generator of homology spheres}

Given an artinian algebra $A = \frac{R}{I}$, the top nonzero degree of $A$ is called the {\bf socle degree} of $A$. Moreover, the vector space 
$$
    \Soc(A) = \{f \in A \st x_i f = 0 \qforevery i\}
$$
is called the {\bf socle} of $A$. If $\Soc(A) = A_{d}$, where $d$ is the socle degree of $A$, we say $A$ is a {\bf level algebra}. \cref{t:macaulayduality} says that when $A = \frac{R}{I}$ is an Artinian Gorenstein (AG) algebra of socle degree $d + 1$, and in particular $\Soc(A) = A_{d + 1}$, the inverse system of $I$ is a cyclic module generated by a polynomial $F \in S$ of degree $- (d + 1)$. The polynomial $F$ is called the {\bf Macaulay dual generator} of $A$.
 
\cref{t:gorensteinequivalences} implies that for a field $\K$, and a $\K$-homology sphere $\Delta$, the ring $\Kco$ is an AG algebra, and hence has a Macaulay dual generator. In the general case, where $\Delta$ is a $d$-dimensional complex and $\Kco(\Delta)$ is simply an artinian algebra, we call the generators of $(I_\Delta, e_1, \dots, e_{d+1})^{-1}$ the {\bf top coinvariant stresses} of $\Delta$.
Note that when $\Delta$ is a homology sphere with unique top coinvariant stress $F_\Delta$, we have $F_\Delta \in S$, but since $S \cong R$, we may view $F_\Delta$ as a polynomial in $R$ via the map $y_i \mapsto x_i$. In this setting, we conclude $\Soc(\Kco) \cong \{c F_\Delta \st c \in \K\}$.  

Our goal in this section is to give an explicit formula for $F_\Delta$, given an arbitrary $\K$-homology sphere $\Delta$. We start with a classical fact from algebraic combinatorics, which in our setting translates to computing $F_\Delta$, where $\Delta$ is the boundary of the simplex.

\begin{example}[{\bf Vandermonde determinants and the boundary of a simplex}]
    Let $\Delta$ be the boundary of the simplex on $n$ vertices and $I_\Delta = (x_1 \dots x_n) \subset R$ its Stanley-Reisner ideal. Then $\Kco(\Delta)$ in this case is exactly the coinvariant ring of the symmetric group on $n$ elements, defined by the ideal generated by elementary symmetric polynomials $(e_1, \dots, e_n)$. As a consequence, $\Kco$ is a complete intersection of socle degree $\binom{n}{2}$. It is known in this case that the top coinvariant stress of $\Kco$ (which is also the Macaulay dual generator of the coinvariant ring of $S^n$) is the polynomial:
    $$
        V([n]) = \prod_{1 \leq i < j \leq n} (x_i - x_j).
    $$
\end{example}

Let $R = \K[x_1, \dots, x_n]$ and $B = \{i_1, \dots, i_s\} \subset [n]$. The {\bf Vandermonde determinant} on $B$ is the polynomial 
$$
    V(B) = \prod_{1 \leq i_j < i_k \leq n}(x_{i_j} - x_{i_k}).
$$
By~\cref{t:gorensteinequivalences,t:macaulayduality}, if $\Delta$ is not a cone, $\Delta$ has a unique top coinvariant stress if and only if $\Delta$ is a Cohen-Macaulay orientable pseudomanifold. The key idea of the main result of this section, is that the top coinvariant stress $F_\Delta$ should contain information on an orientation of $\Delta$.

\begin{theorem}[{\bf Top homology and top coinvariant stresses}]\label{l:symmetriccirc}
    Let $\Delta = \tuple {F_1, \dots, F_s, G_1, \dots, G_m}$ be a $d$-dimensional simplicial complex where $F_1, \dots, F_s$ are the $d$-dimensional facets of $\Delta$, assume $\tilde H_d(\Delta; \K) \neq 0$ and let $F = c_1 F_1 + \dots + c_s F_s$ be a nonzero element in $H_d(\Delta; \K)$. Then $F_\Delta \in \Sco_{\binom{d + 2}{2}}(\Delta)$, where 
    \begin{align*}
        F_\Delta & = c_1x_{F_1}V(F_1) + \dots + c_s x_{F_s}V(F_s) \\
                 & = c_1 \sum_{\sigma \in S^{F_1}} \sign(\sigma) x_{\sigma(F_1)}^{\mu} + \dots + c_s \sum_{\sigma \in S^{F_s}}\sign(\sigma) x_{\sigma(F_s)}^\mu,
    \end{align*}
    $\mu = (d + 1, \dots, 1)$ and $S^{F_i}$ is the symmetric group on the elements of $F_i$.
    
    In particular, $\dim \Sco_{\binom{d + 2}{2}}(\Delta) > 0$.
\end{theorem}

\begin{proof}
    In order to prove $F_\Delta \in \Sco_{\binom{d + 2}{2}}(\Delta)$, we need to show $f \circ F_\Delta = 0$ for every generator $f$ of $(I_\Delta, e_1, \dots, e_{d + 1})$. Note that the support of every monomial in $F_\Delta$ is a face of $\Delta$, hence $m \circ F_\Delta = 0$ for every monomial generator $m$ of $(I_\Delta, e_1, \dots, e_{d+1})$. In particular, for every monomial $x_g$ of $e_k$, where $g$ is not a face of $\Delta$, we have $x_g \circ F_\Delta = 0$, so we may disconsider these monomials.

    Proving $e_k \circ F_\Delta = 0$ is equivalent to proving the coefficient of $m$ in $e_k \circ F_\Delta$ is $0$ for every monomial $m$ of degree $\binom{d + 2}{2} - k$. We have three cases:
    \begin{enumerate}
        \item If the support of $m$ has less than $d$ elements, the coefficient must be zero, since every monomial in $F_\Delta$ has only one variable to the power $1$, and $e_k$ consists only of squarefree monomials.
        \item If the support of $m$ has $d + 1$ elements, we may assume without loss of generality that the support of $m$ is $F_1$. The only monomials in $F_\Delta$ that contribute to the coefficient of $m$ in $e_k \circ F_\Delta$ are the monomials in $c_1 x_{F_1} V(F_1)$. Moreover, since every monomial in the support of $e_k$ is squarefree, the number of variables in the support of $m$ with the same power is $2a$ for some $a \geq 1$. Since there is an even number of variables with the same power in the support of $m$, there is also an even number of ways to obtain the monomial $m$ from the operation $x_\tau \circ F_\Delta$, where $\tau$ is a $(k-1)$--face of $\Delta$. The coefficient of $m$ in $e_k \circ F_\Delta$ is
        $$
            c_1\sum_{j = 1}^a (\sign(\sigma_j) - \sign(\sigma_j))= 0,
        $$
        where the signs of permutations come in pairs since if $x_i$ and $x_j$ have the same power in $m$, then if $\sigma$ is a permutation such that $x_g \circ x_{\sigma(F_1)}^\mu = m$ for some monomial $x_g$ of $e_k$, there exists another monomial $x_{g'}$ of $e_k$ such that $x_{g'} \circ x_{(ij) \circ \sigma(F_1)}^\mu = m$, and $\sigma$ and $(ij) \circ \sigma$ have opposite signs.
        % \textcolor{red}{finish argument}
        \item If the support of $m$ has $d$ elements, assume $m = x_{i_1}^{a_1} \dots x_{i_d}^{a_d}$, where $a_1 + \dots + a_d = \binom{d + 2}{2} - k$ and let $\{F_{j_1}, \dots, F_{j_r}\}$ be the facets of $\Delta$ that contain the ridge $\tau = \{i_1, \dots, i_d\}$. Let $z_l$ be such that $\{i_1, \dots, i_d\} \cup z_l = F_l$. The collection of monomials that contribute to the coefficient of $m$ is 
        $$
            \{m x_\lambda x_{z_{j_1}} \st \lambda \subset \tau \qwhere |\lambda| = k - 1 \} \cup \dots \cup \{m x_\lambda x_{z_{j_r}} \st \lambda \subset \tau \qwhere |\lambda| = k - 1\}.
        $$
        In particular, for every $\lambda$ there exists a permutation $\sigma_\lambda$ such that the coefficient of $m$ in $e_k \circ F_\Delta$ is 
        $$
            \sum_{\lambda} \sign(\sigma_{\lambda}) \pi_\tau \partial F,
        $$
        where $\partial: C_d \to C_{d-1}$ is the boundary map of $\Delta$, and $\pi_\tau: C_{d-1} \to \{c \tau \st c \in \K\}$ is the projection from $C_{d-1}$ to the span of $\tau$ inside $C_{d-1}$. By assumption $F$ is a $d$-cycle of $\Delta$, hence $\partial F = 0$.
    \end{enumerate} 
\end{proof}
 
\begin{remark}
    In the proof of~\cref{l:symmetriccirc}, case $(2)$ can be seen as case $(3)$ applied to the boundary of a simplex of higher dimension. The key point being the fact that the map $e_k \circ$ mimics the boundary map of $\Delta$ when it is applied to polynomials of the form $\sum x_{F_i} V(F_i)$.
\end{remark}

An important consequence of~\cref{l:symmetriccirc} is the following. 
 
\begin{corollary}[{\bf Top betti numbers and $\textstyle\binom{d + 2}{2}$--coinvariant stresses}]
    Let $\Delta$ be a $d$-dimensional simplicial complex and $t = \binom{d + 2}{2}$. Then 
    $$
        \dim \Sco_{t}(\Delta) \geq \dim \tilde H_{d}(\Delta; \K).
    $$
    Moreover, assuming $\K$ is a field of characteristic zero and $\Delta$ is Cohen-Macaulay, equality holds.
\end{corollary}

\begin{proof}
    Let $F_1, \dots, F_s$ be the $d$-dimensional facets of $\Delta$. Since the support of every monomial in $x_{F_i} V(F_i)$ is $F_i$, we conclude the polynomials 
    $$
        x_{F_1} V(F_1), \dots, x_{F_s} V(F_s)
    $$
    are linearly independent. Applying~\cref{l:symmetriccirc}, a basis of $\tilde H_d(\Delta; \K)$ gives us a linearly independent set of $\Sco_t(\Delta)$ of size $\dim \tilde H_d(\Delta; \K)$.

    For the last part of the statement, note that for a field $\K$ of characteristic zero and a Cohen-Macaulay complex (over $\K$), $\dim \tilde H_d(\Delta; \K) = h_{d+1}$. The result then follows since by~\cref{t:macaulayduality} $\dim \Sco_t(\Delta)$ is the coefficient of $t = \binom{d + 2}{2}$ in the polynomial
    $$
        \HS_{\Kco(\Delta)}(q) = h_\Delta(q)[d]_q!,
    $$
    which is equal to $h_{d+1}$.
\end{proof}

In the special case where $\Delta$ is a homology sphere, we have the following direct consequence of~\cref{l:symmetriccirc}.

\begin{corollary}[{\bf The top coinvariant stress of a $\K$-homology sphere}]\label{t:universalform}
    Let $\Delta = \tuple{F_1, \dots, F_s}$ be a $d$-dimensional $\K$-homology sphere with orientation $\e$. Then the top coinvariant stress of $\Delta$ is given by 
    $$
        F_\Delta = \e(F_1)x_{F_1}V(F_1) + \dots +\e(F_s)x_{F_s}V(F_s).
    $$
\end{corollary}

\begin{proof}
    Since $\Delta$ is a Cohen-Macaulay orientable $d$-dimensional pseudomanifold, the ring $S/I_\Delta$ is Gorenstein and (since $e_1, \dots, e_{d+1}$ forms a regular sequence) $\Kco(\Delta)$ is also a Gorenstein ring. In particular, the socle of $\Kco(\Delta)$ is a $1$-dimensional vector space, and by~\cref{l:basicHS}, the socle degre of $\Kco(\Delta)$ is 
    $$  
        d + 1 + \binom{d + 1}{2} = \binom{d + 2}{2}.
    $$ 
    The result then follows by~\cref{l:symmetriccirc}.
\end{proof}

% \begin{corollary}\label{c:universalform}
%     If $\Delta$ is a $d$-dimensional complex, $\K$ a field and $\tilde H_d(\Delta; \K) \neq 0$, then 
%     $$
%         \dim \Sco_{\binom{d + 2}{2}}(\Delta) = \dim (\Kco(\Delta))_{\binom{d + 2}{2}} > 0.
%     $$
% \end{corollary}

% \begin{proof}
%     The proof follows directly since the polynomial $F_\Delta$ defined in~\cref{l:symmetriccirc} is still defined, even though $\Delta$ is not necessarily a $\K$-homology sphere, and $F_\Delta \in \Sco_{\binom{d + 2}{2}}(\Delta)$ again by~\cref{l:symmetriccirc}.
% \end{proof}

\begin{example}[{\bf Lack of orientability}]
    Let $\Sigma$ be the triangulation of $\R \Pp^2$ from~\cref{e:projective} and $\K$ a field of characteristic zero. Then $I_\Sigma + (e_1, e_2, e_3)$ has $6$ top coinvariant stresses, all of degree $5$. In particular
    $$
        \dim \Sco_6(\Sigma) = \dim \Kco(\Sigma)_6 = 0.
    $$
\end{example}

\begin{example}[{\bf Non-CM orientable pseudomanifolds}]
    Let $\Lambda$ be the triangulation of the pinched torus from~\cref{e:pinched}. Then $\Lambda$ has four top coinvariant stresses of degrees $3,4,5,6$. The top coinvariant stress of degree $6$ is the polynomial from~\cref{l:symmetriccirc}.
\end{example}

% \section{Lefschetz properties of the universal system of parameters of Gorenstein* complexes}\label{s:lefschetz}

\section{Lefschetz properties in the monomial setting and inverse systems}\label{s:wlpmoninverse}

In the previous section we provided a formula for the unique top coinvariant stress of a homology sphere $\Delta$. It turns out that the same polynomial appears in a different setting, and has been studied again in the case where $\Delta$ is the boundary of a simplex. The goal of the next sections is to develop these connections. We begin by introducing the required notions and their motivations. 

The main idea in Stanley's proof of the $g$-theorem for simplicial polytopes, that also plays a central role in the recent proofs of the $g$-theorem for simplicial spheres, is the use of Lefschetz properties.
Let $A = R/I$ be an artinian graded algebra of socle degree $d$ and $L \in R_1$ a linear form.  We say $L$ is a {\bf weak Lefschetz element} of $A$ if the multiplication maps $\times L: A_i \to A_{i + 1}$ have full rank for every $i < d$. The linear form $L$ is said to be a {\bf strong Lefschetz element} if the multiplication maps $\times L^j: A_{i} \to A_{i + j}$ have full rank for every $i,j$ such that $i + j \leq d$.
If $A$ has a weak Lefschetz element (resp. strong Lefschetz element), we say $A$ has the {\bf weak Lefschetz property} (resp. {\bf strong Lefschetz property}), which we abbreviate to WLP (resp. SLP).
Note that if an algebra $A$ has the WLP, the coefficients of its Hilbert series form a unimodal sequence.

The study of Lefschetz properties of artinian algebras started with the seminal paper of Stanley~\cite{S80}, where he showed the following result.

\begin{theorem}[{\bf Monomial complete intersections and the SLP}, \cite{S80}]\label{t:monomialciWLP}
    Let $I = (x_1^{a_1}, \dots, x_n^{a_n}) \subset R = \K[x_1, \dots, x_n]$ where $\K$ is a field of characteristic zero. Then the algebra $A = \frac{R}{I}$ has the SLP.
\end{theorem}

\cref{t:monomialciWLP} shows that an important family of algebras has the SLP. Since then, several different classes of algebras have been studied with respect to their Lefschetz properties. In Stanley's proof of~\cref{t:monomialciWLP}, the technique used was the Hard Lefschetz theorem from algebraic topology. Over the years, these properties have been studied from a plethora of areas, including algebraic and differential geometry, commutative algebra and combinatorics. 

In this section, we explore the connections of inverse systems and Lefschetz properties of monomial ideals, and in particular, the consequences of the $g$-theorem in this field. We then apply the notion of coinvariant stresses to the failure of Lefschetz properties of monomial algebras.

We note that in~\cite[Proposition 2.2]{MMN2011}, the authors observed that when $A$ is defined by a monomial ideal, the linear form $L = x_1 + \dots + x_n$ is a weak (resp. strong) Lefschetz element if and only if $A$ has the WLP (resp. SLP). 

One of the first examples of a level monomial algebra failing the WLP is due to Brenner and Kaid~\cite{BK2007}, where they used algebraic geometric techniques to conclude that the algebra 
$$
    A = \frac{\K[x,y,z]}{(xyz, x^3, y^3, z^3)}
$$
fails the WLP. This example was then generalized by Migliore, Miró-Roig and Nagel~\cite{MMN2011}, where the authors used liaison theory to show the following.

\begin{theorem}[{\bf Monomial almost complete intersections and the SLP}, \cite{MMN2011}]\label{t:monomialaciWLP}
    The algebra 
    $$
        A = \frac{\K[x_1,\dots, x_n]}{(x_1\dots x_n,x_1^n, \dots, x_n^n)}
    $$
    fails the WLP.
\end{theorem}

There are two key steps in the proof of~\cref{t:monomialaciWLP}. For some special $i$:
\begin{enumerate}
    \item[] {\bf Step 1.} There exists a polynomial $F$ in the kernel of the map $\times L^T: A_i \to A_{i - 1}$; and \label{en:1}
    \item[] {\bf Step 2.} $\HF(i, A) \leq \HF(i - 1, A)$. \label{en:2}
\end{enumerate}

We now show that for a monomial algebra $A = \frac{R}{I}$, where $I = (x_1^{a_1}, \dots, x_n^{a_n}) + I_\Delta$ for some simplicial complex $\Delta$, step 1 is intrinsically related to studying sops of $I_\Delta + L$.

\begin{lemma}\label{l:samematrix}
    Let $f \in R = \K[x_1, \dots, x_n]$, $S = \K[y_1, \dots, y_n]$ where $R$ acts on $S$ by contraction, and $J \subset R$ a monomial ideal. Then the matrix $M$ that represents the map 
    $$
    \times f^T: R_i \to R_{i - 1} 
    $$
    is the same matrix that represents the map 
    $$
    f \circ : S_{-i} \to S_{-i + 1}
    $$
    where $\circ$ denotes contraction. 

    In particular, the matrix that represents the map
    $$
        \times f^T: (R/J)_i \to (R/J)_{i - 1}
    $$
    is obtained from $M$ by deleting rows and columns corresponding to monomials in $J$.
\end{lemma}

\begin{proof}
    Let $F = \sum a_j m_j \in R$ where $m_j$ is a monomial of degree $i - 1$ and $a_j \in \K$ for every $j$. We know multiplication by $x_k$ corresponds to increasing the power of $x_k$ in every $m_j$ by $1$. The transpose of this map then satisfies:
    $$
        \times x_k^T(x_1^{a_1}\dots x_n^{a_n}) = 
        \begin{cases}
            x_1^{a_1} \dots x_k^{a_k - 1} \dots x_n^{a_n} \quad \mbox{ if $a_k > 0$} \\
            0 \quad \mbox{ otherwise}    
        \end{cases}.
    $$
    In particular, since $\times f^T$ is a composition and sum of the maps above, given the isomorphism $\varphi: R \to S$ where $\varphi(x_j) = y_j$, we have $\varphi \circ \times f^T$ is the same map as $f \circ: S_{-i} \to S_{-i + 1}$. Noticing that the matrix that represents the map $\varphi$ is just the identity matrix, the result follows.

    The last part of the statement follows since taking a quotient of $R$ by a monomial ideal $J$ does not add any relation besides removing elements from the basis of each graded piece of $R$.
\end{proof}

We now state the main consequence of~\cref{l:samematrix}.

\begin{theorem}[{\bf Lefschetz properties of monomial ideals via inverse systems}]\label{t:wlpinversesystems}
    Let 
    $$  
        J = I + (x_1^{a_1},\dots, x_n^{a_n}) \subset R = \K[x_1, \dots, x_n]
    $$
    where $I$ is a monomial ideal, $L = x_1 + \dots + x_n$ and $A = \frac{R}{J}$. Let $S = \K[y_1,\dots, y_n]$ be a polynomial ring where $R$ acts on $S$ by contraction. Then a multiplication map 
    $$
        \times L: A_{i - 1} \to A_{i}
    $$
    fails to be surjective if and only if there exists a sequence of elements $(f_1, \dots, f_d)$ such that $I + (f_1, \dots, f_d, L)$ is an artinian ideal, and a polynomial $G \in (I + (L, f_1, \dots, f_d))^{-1}_{-i}$ such that 
    $$
        G = \sum_j c_{j} y_1^{b_{j1}}\dots y_n^{b_{jn}},
    $$
    $c_j \in \K$ and $b_{jk} < a_{k}$ for every $j, k$.
\end{theorem}

\begin{proof}
    The multiplication map $\times A_{i - 1} \to A_i$ fails to be surjective if and only if its transpose fails to be injective. By~\cref{l:samematrix}, a nonzero polynomial $G$ in the kernel of $\times L^T: A_i \to A_{i - 1}$ exists if and only if there exists a polynomial $G'$ satisfying $m \circ G' = L \circ G' = x_i^{a_i} \circ G' = 0$ for every $i$, and every generator $m$ of $I$. Taking any such $G'$ and the Gorenstein artinian ideal $P$ such that $P^{-1} = (G')$, we conclude $I + L \subset P = (f_1, \dots, f_d)$, hence the generating set $(f_1, \dots, f_d)$ of $P$ can be taken as the sequence of elements from the statement. The assumption on the powers of variables in the monomials appearing in $G$ guarantees $x_i^{a_1} \in P$ for every $i$. 
\end{proof}

\cref{t:wlpinversesystems} gives a new perspective as to why do monomial ideals fail the WLP. A special case, which is the main motivation for the next section, is the case where $I$ is a monomial artinian reduction of a Cohen-Macaulay squarefree monomial ideal $I_\Delta$, and the sequence of elements $f_1, \dots, f_d$ from~\cref{t:wlpinversesystems} is a sop of $I_\Delta$.
 
Our next results exemplify how this approach can be used to translate results from the generic to the monomial setting. We first give the necessary definition (see~\cite{MN2013} and~\cite[p. 3]{MNZ2024} for more details on~\cref{d:stacked}).

\begin{corollary}\label{c:1}
    Let $\Delta$ be a $(d-1)$--dimensional simplicial sphere, $L = x_1 + \dots + x_n$, $\theta_1,\dots, \theta_{d - 1}$ general linear forms and $J = I_\Delta + (x_1^{a_1},\dots, x_n^{a_n})$ where $a_i \geq d + 1$ for every $i$. Let $K^i_L$ denote the kernel of the map $\times L^T: (\frac{R}{J})_{i} \to (\frac{R}{J})_{i - 1}$ and $K^i_j$ the kernel of the map $\times \theta_j^T: (\frac{R}{J})_i \to (\frac{R}{J})_{i - 1}$. Then 
    $$
        \dim K^i_L \cap K^i_1 \cap \dots \cap K^i_{{d-1}} > 0 \qfor i \leq d.
    $$
    In particular, the multiplication maps $\times L: (\frac{R}{J})_{i - 1} \to (\frac{R}{J})_{i}$ are not surjective for $i \leq d$.
\end{corollary}

\begin{proof}
    The set of zero divisors of $\frac{R}{I_\Delta}$ is given by the union of associated primes of $I_\Delta$, hence if $I_\Delta \neq (x_1, \dots, x_n)$, $L$ is not a zero divisor of $\frac{R}{I_\Delta}$. The ring $\frac{R}{I_\Delta + L}$ is a standard graded algebra of dimension $d - 1$, so a set of general linear forms $\theta_1, \dots, \theta_{d-1}$ is a lsop, let $A = \frac{R}{I_\Delta + (L, \theta_1, \dots, \theta_{d-1})}$. By~\cref{l:basicHS} and since $\Delta$ is a simplicial sphere, we know $\dim A_i = h_i > 0$ for $i \leq d$ and $\dim A_i = 0$ for $i > d$. In particular, $x_j^{d + 1} \in I_\Delta + (L, \theta_1, \dots, \theta_{d-1})$ for every $j$. Let $G \in (I_\Delta + (L, \theta_1, \dots, \theta_{d-1}))^{-1}_{-i}$, where $i \leq d$. By~\cref{t:wlpinversesystems,l:samematrix}, there exists $G'$ such that $G' \in K^i_L, K^i_1, \dots, K^i_d$. The last statement follows since the transpose of the maps are not injective.  
\end{proof}

\begin{definition}\label{d:stacked}
    Let $\Delta$ be a simplicial $(d-1)$--sphere and $1 \leq k \leq \frac{d}{2} - 1$. Then $\Delta$ is said to be {\bf $k$-stacked} if $g_{k + 1} = 0$.
\end{definition}

\begin{corollary}\label{c:2}
    Let $\Delta$ be a $(d-1)$--simplicial sphere that is not $k$-stacked for some fixed $k \leq \frac{d}{2} - 1$, and $J = I_\Delta + (x_1^{a_1},\dots, x_n^{a_n})$ where $a_i > k + 1$. Then there exists $d+1$ linear forms $\theta_1,\dots, \theta_{d+1}$ such that 
    $$
        \dim K_1^{k + 1} \cap \dots \cap K_{d+1}^{k + 1} > 0,
    $$
    where $K_j^{k + 1}$ is the kernel of the map $\times \theta_j^T: (\frac{R}{J})_{k + 1} \to (\frac{R}{J})_{k}$.
\end{corollary}

\begin{proof}
    The proof of the $g$-theorem for simplicial spheres~\cite{A2018} implies there exists a lsop $\theta_1, \dots, \theta_{d}$ of $\frac{R}{I_\Delta}$, with a strong Lefschetz element $\theta_{d+1}$ of $A = \frac{R}{I_\Delta + (\theta_1, \dots, \theta_d)}$. The SLP of $A$ then implies  
    $$
        \dim (\frac{A}{\theta_{d+1}})_{k + 1} = h_{k + 1} - h_{k} = g_{k + 1} > 0,
    $$
    where the last inequality follows because of the condition on $k$ and since $\Delta$ is not a $k$-stacked sphere. 
    Let 
    $$
        T = I_\Delta + (\theta_1, \dots, \theta_{d+1}) \qand F \in T^{-1}_{-k-1}.
    $$
    By definition we know $\theta_i \circ F = 0$ for every $i = 1, \dots, d + 1$.~\cref{l:samematrix} then implies there exists a polynomial $F'$ in the kernel of $\times \theta_i^T: (\frac{R}{J})_{k + 1} \to (\frac{R}{J})_{k}$ for every $i$.
\end{proof}

Although it may seem as if~\cref{c:1,c:2} can be used to prove failure of WLP of monomial ideals, this turns out to rarely be the case. As was pointed out previously, in order to prove failure of WLP using these results, we would have to show the Hilbert function of the algebra is decreasing at the same degree as the element in the kernel, which turns out to rarely be true. In fact, since the algebras in~\cref{c:1,c:2} are level monomial algebras, a result due to Hausel~\cite[Theorem 6.2]{H2005} says the multiplication maps by a general linear form are injective up to half of the socle degree.

Instead,~\cref{c:1,c:2} show us one way to translate results about regular sequences of monomial ideals and Lefschetz properties of (linear) artinian reductions of monomial ideals, to the monomial setting.

\begin{example}\label{e:sphereinversewlp}
    Consider the simplicial sphere $\Gamma$ from~\cref{e:pinched}. The geometric realization of $\Gamma$ is pictured in~\cref{e:pinched}, and since  $|\Gamma|$ is the convex hull of its vertices, Macaulay2 computations imply the linear forms
    %the connections between toric varieties and simplicial polytopes~\cite[Chapter 12]{toricbook} imply the linear forms
    \begin{align*}
        \theta_1 & = \frac{1}{2}x_{2}+\frac{1}{2}x_{3}-\frac{1}{2}x_{4}-\frac{1}{2}x_{5}-\frac{1}{2}x_{6}-\frac{1}{2}x_{7}+\frac{1}{2}x_{8}+\frac{1
       }{2}x_{9} \\
       \theta_2 & = -x_{0}+x_{1}+\frac{1}{2}x_{2}+\frac{1}{2}x_{3}+\frac{1}{2}x_{4}+\frac{1}{2}x_{5}-\frac{1}{2}x_{6}-\frac{1}{2}x_{7}-\frac{1}{2}x
       _{8}-\frac{1}{2}x_{9} \\
       \theta_3 & = \frac{1}{2}x_{2}-\frac{1}{2}x_{3}-\frac{1}{2}x_{4}+\frac{1}{2}x_{5}-\frac{1}{2}x_{6}+\frac{1}{2}x_{7}+\frac{1}{2}x_{8}-\frac{1
       }{2}x_{9}
    \end{align*}
    form an lsop of $I_\Gamma$, where the linear forms are obtained via the construction in~\cref{eq:plinear}. Abusing notation, the Macaulay dual generator of $I_\Gamma + (\theta_1, \theta_2, \theta_3)$ is the polynomial $F_1 =   
    -2\,x_{0}^{3}-2\,x_{1}^{3}+x_{1}^{2}x_{2}+x_{2}^{3}+x_{1}^{2}x_{3}-x_{1}x_{2}x_{3}+x_{2}x_{3}^{2}+x_{1}^{2}x_{4}-x_{1}
      x_{3}x_{4}+x_{3}x_{4}^{2}+x_{1}^{2}x_{5}-x_{1}x_{2}x_{5}-x_{1}x_{4}x_{5}+x_{2}x_{5}^{2}+x_{4}x_{5}^{2}-x_{5}^{3}+x_{0
      }^{2}x_{6}+x_{3}^{2}x_{6}-x_{3}x_{4}x_{6}+x_{3}x_{6}^{2}+x_{0}^{2}x_{7}+x_{2}^{2}x_{7}+x_{4}^{2}x_{7}-x_{2}x_{5}x_{7}-x
      _{4}x_{5}x_{7}+x_{5}^{2}x_{7}-x_{0}x_{6}x_{7}-x_{4}x_{6}x_{7}+x_{6}^{2}x_{7}+x_{2}x_{7}^{2}+x_{4}x_{7}^{2}-x_{5}x_{7}^{
      2}+x_{7}^{3}+x_{0}^{2}x_{8}-x_{2}^{2}x_{8}-x_{0}x_{7}x_{8}-x_{2}x_{7}x_{8}+x_{2}x_{8}^{2}+x_{7}x_{8}^{2}-x_{8}^{3}+x_{0
      }^{2}x_{9}+x_{2}^{2}x_{9}-x_{2}x_{3}x_{9}-x_{0}x_{6}x_{9}-x_{3}x_{6}x_{9}-x_{0}x_{8}x_{9}-x_{2}x_{8}x_{9}+x_{8}^{2}x_{9
      }+x_{2}x_{9}^{2}+x_{6}x_{9}^{2}$, and using Macaulay2~\cite{M2} we see that $F_1 \in K_1^3 \cap K_2^3 \cap K_3^3$ where $K_i^3$ is the kernel of the map
      $$
        \times \theta_i^T: \Big{(}\frac{R}{I_\Gamma + (x_i^4 \st 0 \leq i \leq 9)}\Big{)}_3 \to \Big{(}\frac{R}{I_\Gamma + (x_i^4 \st 0 \leq i \leq 9)}\Big{)}_2.
      $$

    Similarly, using Macaulay2~\cite{M2} one can check that the linear forms 
    \begin{align*}
        L & = x_0 + x_1 + x_2 + x_3 + x_4 + x_5 + x_6 + x_7 + x_8 + x_9 \\
        \theta_4 & = x_{0}+x_{1}+x_{2}-x_{3}+x_{4}-x_{5}-x_{6}+x_{7}-x_{8}+x_{9}\\
        \theta_5 & = -x_{1}+x_{2}+x_{5}-x_{6}-x_{7}+x_{8}-x_{9}
    \end{align*} 
    form an lsop of $I_\Gamma$. Abusing notation one more time, the Macaulay dual generator of $I_\Gamma + (L, \theta_4, \theta_5)$ is the polynomial $F_2 = 
    -3\,x_{1}^{3}-x_{1}^{2}x_{2}+x_{1}x_{2}^{2}+2\,x_{1}^{2}x_{3}+2\,x_{1}x_{2}x_{3}+4\,x_{3}^{3}+4\,x_{1}^{2}x_{4}-4\,x_{1}x_{3}x
       _{4}-4\,x_{3}^{2}x_{4}-4\,x_{1}x_{4}^{2}+4\,x_{3}x_{4}^{2}-4\,x_{4}^{3}-2\,x_{1}^{2}x_{5}-2\,x_{1}x_{2}x_{5}+4\,x_{1}x_{4}x_{5}-
       4\,x_{3}^{2}x_{6}+4\,x_{3}x_{4}x_{6}-4\,x_{4}^{2}x_{6}+4\,x_{3}x_{6}^{2}-4\,x_{4}x_{6}^{2}-4\,x_{6}^{3}-8\,x_{0}^{2}x_{7}+8\,x_{
       4}^{2}x_{7}+2\,x_{2}x_{5}x_{7}-4\,x_{4}x_{5}x_{7}-4\,x_{0}x_{6}x_{7}+4\,x_{4}x_{6}x_{7}+8\,x_{0}x_{7}^{2}-8\,x_{4}x_{7}^{2}+2\,x
       _{5}x_{7}^{2}+4\,x_{0}x_{7}x_{8}-2\,x_{2}x_{7}x_{8}-2\,x_{7}^{2}x_{8}+8\,x_{0}^{2}x_{9}-x_{2}^{2}x_{9}-2\,x_{2}x_{3}x_{9}+4\,x_{
       3}^{2}x_{9}+4\,x_{0}x_{6}x_{9}-4\,x_{3}x_{6}x_{9}+4\,x_{6}^{2}x_{9}-4\,x_{0}x_{8}x_{9}+2\,x_{2}x_{8}x_{9}-8\,x_{0}x_{9}^{2}+x_{2
       }x_{9}^{2}+2\,x_{3}x_{9}^{2}-4\,x_{6}x_{9}^{2}+2\,x_{8}x_{9}^{2}+7\,x_{9}^{3}$. By the same argument, $F_2$ is in the kernel of the map 
    $$
        \times L^T: \frac{R}{I_\Gamma + (x_i^4 \st 0 \leq i \leq 9)} \to \frac{R}{I_\Gamma + (x_i^4 \st 0 \leq i \leq 9)}.
    $$ 
    In particular, if $A = \frac{R}{I_\Gamma + (x_i^4 \st 0 \leq i \leq 9)}$ has the WLP, it must be that $\HF_A(2) \leq \HF_A(3)$. 
\end{example}

As we will soon see, even though the algebra $A$ from~\cref{e:sphereinversewlp} satisfies 
    $$
        \HF_A(2) = 34 \leq 74 = \HF_A(3),
    $$
it fails the WLP in later degrees.

We finish this section with a vast generalization of step 1 in the proof of~\cref{t:monomialaciWLP}

\begin{proposition}[{\bf Coinvariant stresses and Lefschetz properties}]\label{p:generalstep1}
    Let $\Delta$ be a $\K$ be a field and $\Delta$ a $d$-dimensional complex such that $\tilde H_d(\Delta; \K) \neq 0$. Set $R = \K[x_1, \dots, x_n]$, $L = x_1 + \dots + x_n \in R$ and $J = I_\Delta + (x_1^{d + 2}, \dots, x_n^{d+2})$. Then the multiplication map 
    $$
        \times L: \Big{(}\frac{R}{J}\Big{)}_{\binom{d+2}{2} - 1} \to \Big{(}\frac{R}{J}\Big{)}_{\binom{d+2}{2}}
    $$
    is not surjective.
\end{proposition}

\begin{proof}
    The result follows by taking the nonzero polynomial $F_\Delta \in \Sco_{\binom{d + 2}{2}}(\Delta)$ from~\cref{l:symmetriccirc}. Notice that $x_i^{d + 2} \circ F_\Delta = 0$ for every $i$, hence we may apply~\cref{t:wlpinversesystems} where the sequence of elements is given by the elementary symmetric polynomials $e_2, \dots, e_{d + 1}$.
\end{proof}

\section{$f$-vector inequalities and the WLP of squarefree monomial ideals}
 
In~\cref{s:wlpmoninverse} we showed that lack of surjectivity of a multiplication map in an artinian monomial algebra is directly related to the existence of special sequences of elements (for example, sops). We then observed that this fact alone is not enough to prove failure of WLP. The crucial missing step (as in step 2 of~\cref{t:monomialaciWLP}) is, for some fixed $i$, the inequality $\HF_A(i - 1) \geq \HF_A(i)$. In~\cite{MMN2011} the authors use liaison theory in order to prove this inequality for monomial almost complete intersections in $n$ variables, and $i = \binom{n}{2}$. The goal of this section is to generalize this inequality, putting this result in the context of $f$-vector inequalities of special classes of simplicial complexes. Results in this section show that liaison theory turns out to be an extremely useful tool to prove the existence of unexpected bijections (see~\cref{l:linkageb}) between different sets of combinatorial objects called compositions, which we will soon define.  It is clear from the plethora of methods that have been previously used in the study of Lefschetz properties of monomial ideals (such as liaison theory in~\cite{MMN2011}), that Lefschetz properties in the monomial setting can be a powerful tool to understand inequalities that $f$-vectors of families of complexes must satisfy.

A {\bf composition} of $n$ with $m$ parts is a vector $(a_1, \dots, a_m)$ of nonnegative integers such that $a_1 + \dots + a_m = n$. Given a simplicial complex $\Delta$ and a monomial ideal $J = I_\Delta + (x_1^{a_1},\dots, x_n^{a_n}) \subset R$, we may view a monomial $M \not \in J$ of degree $i$, as a composition of $i$ with at most $\dim \Delta + 1$ nonzero parts, where exponents correspond to parts. Let $a(n, k, l)$ denote the number of compositions $a = (a_1, \dots, a_l)$ of $n$ with $l$ parts such that $1 \leq a_i \leq k$. We will need the following lemma.

\begin{lemma}\label{l:hilbertcompositions}
    Let $\Delta$ be a simplicial complex of dimension $d$ and $I_\Delta \subset \K[x_1,\dots, x_n] = R$ its Stanley-Reisner ideal. Then the Hilbert function of $J = I_\Delta + (x_1^{k + 2}, \dots, x_{n}^{k+2})$ is given by 
    $$
        \HF_{\frac{R}{J}}(t) = \sum_{i = 0}^d f_i a(t, k + 1, i + 1).
    $$
    In particular we have 
    $$
        \HF_{\frac{R}{J}}(t - 1) - \HF_{\frac{R}{J}}(t) = \sum_{i = 0}^d f_i (a(t - 1, k + 1, i + 1) - a(t, k + 1, i + 1)).
    $$
\end{lemma}

\begin{proof}
    Nonzero monomials in $R/J$ are of the form $m = x_{i_1}^{a_1}\dots x_{i_s}^{a_s}$ where $a_i \leq k + 1$, $\{i_1, \dots, i_s\} \in \Delta$ and $s \leq d + 1$. In other words, for each of the $f_i$ faces of dimension $i$ of $\Delta$, there are $a(t, k + 1, i + 1)$ nonzero monomials in $R/J$ of degree $t$. The result then follows by taking a sum over all faces of dimensions $0 \leq i \leq d$.      
\end{proof}

Let $b(n, k, l)$ be the number of compositions $b = (b_1, \dots, b_l)$ of $n$ with $l$ parts such that $b_i \leq k$.
Subtracting $1$ from every part of a composition $a$ of $n + l$ with $l$ parts satisfying $1 \leq a_i \leq k + 1$, we get a composition $b$ of $n$ with $l$ parts such that $b_i \leq k$. More specifically, we have the following equality 
\begin{equation}\label{eq:equalityab}
    a(n + l, k + 1, l) = b(n, k, l) \qforevery n,k,l.    
\end{equation}

Note that a standard stars and bars argument shows that $b(n, k, l)$ is the coefficient of $x^n$ in the product 
\begin{equation}\label{eq:ciHS}
    (1 + x + \dots + x^k)^l.    
\end{equation}
In particular, since by~\cref{l:basicHS} equation \eqref{eq:ciHS} is the Hilbert series of a codimension $l$ complete intersection generated in degree $k + 1$, the Hilbert function of the monomial complete intersection $I = (x_1^{k + 1}, \dots, x_l^{k + 1}) \subset \K[x_1, \dots, x_l]$ is given by 
\begin{equation}\label{eq:hilbertcomp}
    HF_{\frac{R}{I}}(n) = b(n, k, l).    
\end{equation}
Equations \eqref{eq:equalityab} and \eqref{eq:hilbertcomp} together with~\cref{t:monomialciWLP} imply the following.
\begin{lemma}\label{l:inequalitiescompositions}
    For any $d \geq 1$, the following holds. 
    $$
        \textstyle a(\binom{d + 2}{2} - 1, d + 1, i) \geq a(\binom{d + 2}{2}, d + 1, i) \qfor i \leq  d.
    $$
\end{lemma}
 
\begin{proof}
    First note that by~\eqref{eq:equalityab} we have $a(t, d + 1, j) = b(t - j, d, j)$ and $a(t - 1, d + 1, j) = b(t - j - 1, d, j)$ for any $j > 0$.

    By~\eqref{eq:ciHS} the following equality holds: 
        $$
            HS_{\frac{R}{I}}(q) = b(0, d, i) + b(1, d, i) q +  b(id + i, d, i) q^{id+i},
        $$ 
        where $I = (x_1^{d+1},\dots, x_i^{d + 1}) \subset \K[x_1,\dots, x_i]$. Now since $I$ is a monomial complete intersection, we know the sequence of coefficients of $\HS_{\frac{R}{I}}(q)$ must be symmetric, and by~\cref{t:monomialciWLP}, the sequence is unimodal. In particular, to prove the inequality, we only need to show that the two entries of this sequence we care about appear after its peak. Since the socle degree of $R/I$ is $id + i$, it suffices to show that $\frac{i(d + 1)}{2} \leq t - 1 - i$. Indeed, since $d \geq i$ we have
        \begin{align*}
            \frac{(d + 1)(d + 2)}{2} - \frac{i(d + 1)}{2} - i - 1 & = \frac{(d + 1)(d + 2) - i(d + 3) - 2}{2} \\
            & = \frac{d(d + 3) - i(d + 3)}{2} \\
            & = \frac{(d + 3)(d - i)}{2} \\
            & \geq 0.
        \end{align*} 
        We have shown that 
        $$
            \textstyle a(\binom{d + 2}{2} - 1, d + 1, i) = b(\binom{d + 2}{2} - i - 1, d, i) \geq b(\binom{d + 2}{2} - i, d, i) = a(\binom{d + 2}{2}, d + 1, i) \qfor i \leq d.
        $$
\end{proof}

For the next two lemmas, we need the notion of a basic double linkage from liaison theory. Let $J \subset I \subset R = \K[x_1, \dots, x_n]$ be homogeneous ideals such that $\htt J = \htt I - 1$. Let $\ell \in R$ be a linear form such that $J \colon \ell = \ell$. Then the ideal $I' := \ell I + J$ is called a {\bf basic double link} of $I$. For more details on liaison theory and basic double linkage, see~\cite{KMMNP2001}.

\begin{lemma}[{\cite[Lemma 4.8]{KMMNP2001},\cite[Lemma 2.6]{MMN2011}}]\label{l:linkage}
    Let $J \subset I \subset R$ be homogeneous ideals such that $\htt J = \htt I - 1$, and a linear form $\ell$ such that $I' = \ell I + J$ is a basic double link of $I$. Then the following equality holds for every $j$.
    $$
        \HF_{\frac{R}{I'}}(j) = \HF_{\frac{R}{I}}(j - 1) + \HF_{\frac{R}{J}}(j) - \HF_{\frac{R}{J}}(j - 1).
    $$
\end{lemma}

\begin{lemma}\label{l:linkageb}
    The following equality holds for every $d$ 
    $$
        b\Big{(}\textstyle{\binom{d + 2}{2}} - d - 1, d, d\Big{)} - b\Big{(}\textstyle{\binom{d + 2}{2}} - d, d, d\Big{)} = b\Big{(}\textstyle{\binom{d + 2}{2}} - d - 1, d, d + 1\Big{)} - b\Big{(}{\textstyle \binom{d + 2}{2}} - d - 2, d, d + 1\Big{)}.
    $$
    In particular, for $I_1 = (x_1^{d+1}, \dots, x_d^{d + 1}, x_{d+1}), I_2 = (x_1^{d+1}, \dots, x_d^{d + 1}, x_{d+1}^{d + 1}) \subset R = \K[x_1,\dots, x_{d + 1}]$ we have 
    $$ 
        \HF_{\frac{R}{I_1}}\Big{(}{\textstyle \binom{d + 2}{2}} - d - 1\Big{)} - \HF_{\frac{R}{I_1}}\Big{(}\textstyle {\binom{d + 2}{2}} - d\Big{)} = \HF_{\frac{R}{I_2}}\Big{(}\binom{d + 2}{2} - d - 1\Big{)} - \HF_{\frac{R}{I_2}}\Big{(}\binom{d + 2}{2} - d\Big{)}.
    $$
\end{lemma}

\begin{proof}
    First note that the socle degree of $\frac{R}{I_2}$ is $d(d + 1)$ and since $\binom{d + 2}{2} - d - 1 = \frac{d(d+1)}{2}$, by the symmetry of $\HS_{\frac{R}{I_2}}(t)$ we may rewrite the second statement as 
    $$ 
    \HF_{\frac{R}{I_1}}\Big{(}\textstyle \binom{d + 2}{2} - d\Big{)} - \HF_{\frac{R}{I_1}}\Big{(}\binom{d + 2}{2} - d - 1\Big{)} = \HF_{\frac{R}{I_2}}\Big{(}\binom{d + 2}{2} - d - 1\Big{)} - \HF_{\frac{R}{I_2}}\Big{(}\binom{d + 2}{2} - d - 2\Big{)}
    $$
    The two statements are then equivalent by~\eqref{eq:hilbertcomp}.

    In order to prove the lemma, let $J = (x_1^{d+1}, \dots, x_d^{d+1}), I = J + (x_{d+1}^{d}) \subset R$. Since we have 
    \begin{enumerate}
        \item $\htt J = d$, $\htt I = d + 1$
        \item $J \colon x_{d+1} = J$
    \end{enumerate}
    the ideal $I' := x_{d+1} I + J = (x_1^{d + 1}, \dots, x_{d+1}^{d + 1})$ is a basic double link of $I$. By~\cref{l:linkage} we conclude for every $j$
    $$
        \HF_{\frac{R}{I'}}(j) = \HF_{\frac{R}{I}}(j - 1) + \HF_{\frac{R}{J}}(j) - \HF_{\frac{R}{J}}(j - 1).
    $$
    In particular, since $x_{d + 1}$ is a regular element of $\frac{R}{J}$ we have $\HF_{\frac{R}{J}}(j) - \HF_{\frac{R}{J}}(j - 1) = \HF_\frac{R}{(J, x_{d+1})}(j)$ and so by~\eqref{eq:hilbertcomp} we conclude 
    \begin{equation}\label{eq:extraterm}
        b(j, d, d + 1) - b(j - 1, d, d + 1)  = \HF_{\frac{R}{I}}(j - 1) - \HF_{\frac{R}{I}}(j - 2) + b(j, d, d) - b(j - 1, d, d).        
    \end{equation}
    Finally, the socle degree of $\frac{R}{I}$ is $d(d+1) -1$, which is an odd number and in particular 
    $$
        \HF_{\frac{R}{I}}\Big{(}\frac{d(d+1)}{2}\Big{)} = \HF_{\frac{R}{I}}\Big{(}\frac{d(d+1)}{2} - 1\Big{)}.
    $$
    The result then follows by taking $j = \binom{d + 2}{2} - d$ in equation~\eqref{eq:extraterm}.
\end{proof}

\begin{remark}
    Both~\cref{l:linkageb} and~\cite[Lemma 4.2]{MMN2011} have alternative more elementary proofs, but the liaison theory approach provides exactly the necessary restrictions that make both lemmas useful. This perspective is then useful in order to find such equalities in the first place.
\end{remark}
% The last lemma we need is the following equality on the $f$-vector of pseudomanifolds without boundary, that was pointed out to me by Margaret Bayer.

% \begin{lemma}\label{l:lastentries}
%     Let $\Delta$ be a $d$-dimensional pseudomanifold without boundary. Then 
%     $$
%         2f_{d-1} = (d + 1)f_d.
%     $$
% \end{lemma}

% \begin{proof}
%     We prove the inequality by counting the number of pairs $(R, F)$ where $R$ is a ridge contained on the facet $F$. Since every ridge is contained in exactly 2 facets, there are exactly $2f_{d-1}$ such pairs. Similarly, every facet contains exactly $d + 1$ distinct ridges, hence there are $(d+1)f_d$ such pairs. We then conclude 
%     $$
%         2f_{d-1} = (d + 1)f_{d}.
%     $$ 
% \end{proof}

We are now ready to prove the generalization of step 2 from the proof of~\cref{t:monomialaciWLP}.

\begin{theorem}\label{t:inequalitypseudomanifoldwoboundary}
    Let $\Delta$ be a $d$-dimensional simplicial complex where $d > 0$ and $f_{d - 1} \geq f_d$, $I_\Delta \subset R = \K[x_1,\dots, x_n]$ its Stanley-Reisner ideal and $J = I_\Delta + (x_1^{d + 2}, \dots, x_n^{d + 2})$. Then for $t = \binom{d + 2}{2}$
    $$
        \HF_{\frac{R}{J}}(t - 1) \geq \HF_{\frac{R}{J}}(t).
    $$
    In particular, the inequality holds when $\Delta$ is a $d$-dimensional orientable pseudomanifold without boundary where $d > 0$.
\end{theorem}

\begin{proof}
    By~\cref{l:hilbertcompositions} we know the statement is equivalent to the following inequality:
    \begin{equation}\label{eq:1}
        \sum_{i = 0}^d f_i (a(t - 1, d+1, i + 1) - a(t, d+1, i + 1)) \geq 0.     
    \end{equation}
    By~\cref{l:inequalitiescompositions} we know 
    $$
        a(t - 1, d + 1, i + 1) - a(t, d + 1, i + 1) \geq 0 \qfor i \leq d - 1,
    $$
    hence, since $d > 0$, we may rewrite~\eqref{eq:1} as 
    $$
        C + f_{d - 1}(a(t-1, d+1, d) - a(t, d+1, d)) + f_d(a(t - 1, d+1, d + 1) - a(t, d+1, d + 1)) \geq 0,
    $$
    where both $C$ and the coefficient of $f_{d-1}$ are positive numbers. Our goal then is to show that
    \begin{equation}\label{eq:2}        
        f_{d - 1}(a(t-1, d+1, d) - a(t, d+1, d)) + f_d(a(t - 1, d+1, d + 1) - a(t, d+1, d + 1)) \geq 0.
    \end{equation}
    By~\eqref{eq:equalityab} we may rewrite the left side of~\eqref{eq:2} as 
    \begin{equation}\label{eq:lastone}
        f_{d - 1}(b(t - d - 1, d, d) - b(t - d, d, d)) + f_d(b(t - d - 2, d, d + 1) - b(t - d - 1, d, d + 1)).        
    \end{equation}
    By~\cref{l:linkageb}, equation \eqref{eq:lastone} is equal to 
    $$
        \big{(}b(t - d - 1, d, d) - b(t - d, d, d)\big{)}(f_{d-1} - f_d).
    $$
    By~\cref{l:inequalitiescompositions} and since $\Delta$ satisfies $f_{d - 1} \geq f_d$ we conclude 
    $$
        \HF_{\frac{R}{J}}(t - 1) \geq \HF_{\frac{R}{J}}(t).
    $$
    For the last part of the statement, a $d$-dimensional pseudomanifold without boundary satisfies 
    $$
        f_{d - 1} =\frac{d + 1}{2}f_d,
    $$
    this equality implies $f_{d - 1} \geq f_{d}$ for $d > 0$.
\end{proof}

Under the assumptions of \cref{t:inequalitypseudomanifoldwoboundary}, the transpose of the multiplication map 
$$
    \times L^T: (R/J)_{t} \to (R/J)_{t - 1} \quad \mbox{where } L = x_1 + \dots + x_n \in R
$$
can only have full rank if it is injective. We are now ready to prove our main result, which ties the notion of coinvariant stress from earlier sections, to Lefschetz properties of monomial ideals.
   
\begin{theorem}[{\bf Monomial almost complete intersections revisited}]\label{t:gorensteinfailure}
    Let $\Delta$ be a simplicial complex of dimension $d > 0$ such that $f_{d-1} \geq f_d$ and $\tilde H_d(\Delta; \K) \neq 0$, where $\K$ is a field and $J = I_\Delta + (x_1^{d + 2}, \dots, x_n^{d + 2}) \subset R = \K[x_1,\dots, x_n]$. Then $R/J$ fails the weak Lefschetz property.
\end{theorem}

\begin{proof}
    Let $L = x_1 + \dots + x_n = e_1 \in R$. By~\cref{l:samematrix}, we know if a polynomial $F$ is in the kernel of the map 
    $$
        \times L^T: R_i \to R_{i - 1},
    $$
    and the nonzero monomials of $F$ are not in $J$, then $F$ (as an element of $R/J$) is in the kernel of the map 
    \begin{equation}\label{e:1}
        \times L^T: (R/J)_i \to (R/J)_{i - 1}.        
    \end{equation}
    By~\cref{t:universalform}, there exists a coinvariant stress $F_\Delta$ of $\Delta$ of degree $t = \frac{(d + 1)(d + 2)}{2}$ satisfying $e_1 \circ F_\Delta = 0$. Hence by~\cref{l:samematrix} and the comments above, $F_\Delta$ (as an element of $R/J$) is in the kernel of
    $$
    \times L^T (R/J)_t \to (R/J)_{t - 1}.
    $$
    The result then follows by~\cref{t:inequalitypseudomanifoldwoboundary}.
    % In particular, we only need to show that $\HF(t, J) \leq \HF(t - 1, J)$. Since, the map 
    % $$
    %     \times L^T: (R/J)_{t} \to (R/J)_{t - 1}
    % $$
    % not being injective, implies the multiplication map $\times L: (R/J)_{t - 1} \to (R/J)_t$ is not surjective.

    % By~\cref{l:hilbertcompositions,l:inequalitiescompositions} we have 
    % \begin{align*}
    %     \HF(t - 1, J) - \HF(t, J) & = \sum_{i = 0}^{d} f_i[a(t - 1, d + 1, i + 1) - a(t, d + 1, i + 1)] + f_{d-1}[a(t - 1, d + 1, i + 1) - a(t, d + 1, i + 1)]\\
    %     & = \text{positive terms} + c(f_{d - 1} - f_d) 
    % \end{align*}
    % where $c$ is a positive constant. The result then follows directly by~\cref{l:lastentries} since every Gorenstein* complex is a pseudomanifold without boundary.
\end{proof}

In more algebraic terms,~\cref{t:gorensteinfailure} implies the following.

\begin{corollary}
    Let $I$ be a Gorenstein squarefree monomial ideal of $R = \K[x_1,\dots, x_n]$ such that $\dim \frac{R}{I} = d + 1$ for some $d > 0 $, and for every $i$, $x_i$ divides at least one generator of $I$. Let $J = I + (x_1^{d + 2}, \dots, x_n^{d + 2})$. Then the algebra $\frac{R}{J}$ fails the WLP.
\end{corollary}

\begin{proof}
    Since $I$ is a Gorenstein squarefree monomial ideal, it is the Stanley-Reisner ideal of some complex $\Delta$. Since for every $i$ the variable $x_i$ divides at least one generator of $I$, we conclude $\Delta$ is not a cone over any of its vertices. By~\cref{t:gorensteinequivalences} we conclude $\Delta$ is a homology sphere, and so $d > 0$ implies
    $$  
        f_{d-1} = \frac{d + 1}{2} f_d \geq f_d.
    $$
    The condition on $\dim \frac{R}{I}$ implies we can apply~\cref{t:gorensteinfailure}.
\end{proof}

Note that there exists only one $0$-dimensional orientable pseudomanifold: the simplicial sphere $\Delta$ on $2$ vertices such that $I_\Delta = (xy) \subset \K[x,y]$. Since $A = \frac{\K[x, y]}{(x^2, xy, y^2)}$ is an artinian algebra of codimension $2$, the algebra $A$ must have the SLP by~\cite[Proposition 3.47]{lefschetzbook}.

\begin{example}[{\bf Orientability and Lefschetz properties}]\label{e:wlporientable}
    Let $\Gamma$ be the simplicial sphere from~\cref{e:pinched} and $J = I_\Gamma + (x_0^4, \dots, x_9^4) \subset R$. Computing the Hilbert series of $\frac{R}{J}$ and $\frac{R}{J + L}$ using Macaulay2~\cite{M2} we get the following table.
    $$
    \begin{array}{c|cccccccccccc} 
        i & 0 & 1 & 2 &3 &4 & 5 & 6 & 7 &8 & 9 & 10 & \\
        \hline
        \HF_{\frac{R}{J}}(i) &1&10 & 34& 74& 120& 144& 136& 96& 48 & 16 & 0 \\
        \HF_{\frac{R}{J + L}}(i) & 1 & 9 & 24& 40& 46& 24& 1& 0& 0 & 0 & 0
        \end{array}
    $$
    As predicted by~\cref{t:gorensteinfailure}, $\frac{R}{J}$ fails the WLP in degree $6$. Checking with Macaulay $2$ we see that the ideal $J_1 = I_\Gamma + (x_0^3, \dots, x_9^3)$ has the WLP while the ideal $J_2 = I_\Gamma + (x_0^5, \dots, x_9^5)$ fails the WLP due to injectivity. The same pattern holds when we replace the simplicial sphere $\Gamma$ by the pinched torus $\Lambda$, also from~\cref{e:pinched}.
\end{example}

\begin{example}[{\bf Lack of orientability and Lefschetz properties}]
    Let $\Sigma$ be the triangulation of $\R \Pp^2$ from~\cref{e:projective}.
    
    Let 
    $$
        J = I_\Delta + (x_1^4, \dots, x_6^4) \subset R = \K[x_1,\dots, x_6] \qand L = x_1 + \dots + x_6 \in R
    $$
    where $\K$ is a field of characteristic $0$. Computing the Hilbert series of $R/J$ and $R/(J,L)$ using Macaulay2~\cite{M2}, we get the following table:
    $$
    \begin{array}{c|cccccccccccc} 
        i & 0 & 1 & 2 &3 &4 & 5 & 6 & 7 &8 & 9 & 10 & \\
        \hline
        \HF_{\frac{R}{J}}(i) &1&6 & 21& 46& 75& 90& 85& 60& 30 & 10 & 0 \\
        \HF_{\frac{R}{J + L}}(i) & 1 & 5 & 15& 25& 29& 15& 0& 0& 0 & 0 & 0
        \end{array}
    $$
    In particular, even though~\cref{t:inequalitypseudomanifoldwoboundary} still holds and we see that $90 > 85$, when the characteristic of the base field is $0$, $\Sigma$ is not orientable and $R/J$ satisfies the WLP. 
    
    When the base field has positive characteristic, the situation is more complicated. Even though the ring $R/J$ fails WLP when the characteristic of the base field is $2$, it fails by a lot more than expected. This behaviour can be seen in the table below, where the characteristic of the base field $\K$ of $R$ is $2$. 
    $$
    \begin{array}{c|cccccccccccc} 
        i & 0 & 1 & 2 &3 &4 & 5 & 6 & 7 &8 & 9 & 10 & \\
        \hline
        \HF_{\frac{R}{J}}(i) &1&6 & 21& 46& 75& 90& 85& 60& 30 & 10 & 0 \\
        \HF_{\frac{R}{J + L}}(i) & 1 & 5 & 15& 25& 30& 21& 10& 0& 0 & 0 & 0
        \end{array}
    $$
    Computations involving other characteristics suggest that over characteristic $p > 0$, the ring $R/J$ has the WLP only when $p > 7$.
    Note also that since $\Sigma$ is not a homology sphere, it is not even guaranteed that $h_3 \neq 0$. Indeed, the $h$-vector of $\Sigma$ is $(1,3,6)$, and so (since $\Sigma$ is Cohen-Macaulay when the characteristic of $\K$ is not $2$) the top coinvariant stresses of highest degree of $\Delta$ correspond to elements in the kernel of the map 
    $$
        \times L^T: (R/J)_5 \to (R/J)_4,
    $$
    but since at these degrees the Hilbert function is still increasing, elements in the kernel do not imply failure of the WLP.
\end{example}

\section{Sometimes failure should be expected}

Several classical results about Lefschetz properties say that large families of artinian algebras satisfy the WLP/SLP. When it comes to failure, it is often the case that the failure can be explained by unexpected geometric, combinatorial or algebraic pheonomena. This dichotomy has led to artinian algebras being expected to satisfy the WLP/SLP, and failure being viewed as rare and special behaviour. The goal of this section is to show that in the setting of this paper, failure of WLP should be expected. To the best of our knowledge, this is the first time probabilistic models are being used to prove algebras in a certain family fail the WLP with high probability. For an example where probabilistic models were used to prove that algebras in special families should be expected to have the WLP, see~\cite{NP2024}.
 
In~\cite{LM2006}, Linial and Meshulam introduced a model to study vanishing of the first homology group of $2$-dimensional complexes over a field $\K$. Their model is a natural generalization of the classical Erdős–Rényi~\cite{ER1960} model from graph theory. The general version of the model for higher dimensional complexes complexes, which we now define, was then considered by Meshulam and Wallace in~\cite{MW2009}.

\begin{definition}[{\bf An Erdős–Rényi model for higher dimensional complexes}, \cite{LM2006,MW2009}]\label{d:model}
    Let $\Delta_n$ be the simplex on $n$ vertices, and $\Delta^{(k)}_n$ denote the $k$-skeleton of $\Delta_n$. In other words, $\Delta^{(k)}_n$ is the complex whose facets are the $k$-faces of $\Delta_n$. Let $d \geq 2$ and $Y_d$ be the set of simplicial complexes $Y$ such that $\Delta^{(d-1)}_n \subseteq Y \subseteq \Delta^{(d)}_n$. We denote by $Y_d(n,p)$ the probability space of complexes $Y_d$ with probability measure 
    $$
        \Pp(Y \in Y_d(n, p)) = p^{f_d} (1-p)^{\binom{n}{d + 1} - f_d},
    $$ 
    where $f_d$ is the number of $d$-dimensional faces of $Y$.
\end{definition}
 
One specific question in the study random simplicial complexes, is to give probabilistic criteria for vanishing of homology groups of a random complex. An important example of this type of question, is to find thresholds on the parameter $p$, so that for a fixed field $\K$,  
$$
    \lim_{n \to \infty} \Pp(Y \in Y_d(n, p) \st \tilde H_d(Y; \K) \neq 0) = 1.
$$
 
This question was explored in the original paper~\cite{ER1960}, where the authors study connectivity and the cycle structure of random graphs, and later on by several authors~\cite{LM2006,K2010,ALLTM2013}. Before stating the result we need, we first set some notation from~\cite{ALLTM2013}.

Let $c_1 = 1$, 
$$
    g_d(x) = (d+1)(x+1)e^{-x} + x(1 - e^{-x})^{d+1},
$$
and for $d > 1$ denote by $c_d$ the unique positive solution of the equation $g_d(x) = d + 1$. The inequality $c_d < d + 1$ is true for every $d$. In~\cite{ALLTM2013} the authors show the following result (for $d > 1$), which is an improvement of an earlier result of Koszlov~\cite{K2010}. The case $d = 1$ is a classical result, see for example~\cite{probmethod}.

\begin{theorem}[{\cite[Theorem 1.2]{ALLTM2013}}]\label{t:probtop}
    Let $\K$ be a field and $d \geq 1$. For a fixed $c > c_d$:
    $$
        \lim_{n \to \infty}\Pp(Y \in Y_d(n, \textstyle\frac{c}{n}) \st \tilde H_d(Y; \K) \neq 0) = 1
    $$
\end{theorem}

Let $Y \in Y_d$ and 
$$
    A_Y = \frac{R}{I_Y + (x_1^{d+2},\dots, x_n^{d + 2})},
$$
where $R = \K[x_1, \dots, x_n]$.

Our goal is to apply~\cref{t:probtop} to be able to understand the behaviour of the number 
$$
    \Pp(Y \in Y_d(n, \textstyle\frac{c}{n}) \st A_Y \mbox{ fails the WLP}).
$$

Note that in view of~\cref{t:gorensteinfailure}, this number is connected to~\cref{t:probtop}. Our first result is the following.

\begin{theorem}
    Let $\K$ be a field, $\frac{1}{2} \leq s < 1$ and $d = \lceil sn \rceil < n$. Then for a fixed $0 < p < 1$: 
    $$
        \Pp(Y \in Y_d(n, p) \st A_Y \mbox{ fails the WLP}) \geq \Pp(Y \in Y_d(n, p) \st \tilde H_d(Y; \K) \neq 0).
    $$
\end{theorem}

\begin{proof}
    By the definition of $Y_d$, we know every $Y$ in $Y_d$ contains $\binom{n}{d}$ faces of dimension $d-1$. Denoting by $f_{i}$ the number of $i$-faces of $Y$, we have:
    $$
        \textstyle f_{d} \leq \binom{n}{d + 1} \leq \binom{n}{d} = f_{d - 1},
    $$
    where the second inequality follows from the assumption on $d$ and $s$.
    In particular, by~\cref{t:gorensteinfailure}, for every complex $Y \in Y_d$ such that $\tilde H_d(Y; \K) \neq 0$, the algebra $A_Y$ fails the WLP.
\end{proof}

The main result of this section shows that in the (squarefree) monomial setting, failure of the WLP should (at least in some scenarios) be expected.

\begin{theorem}[{\bf A scenario where failure is expected}]
    Let $d \geq 1$ and fix $c$ such that $c_d < c < d + 1$. Then
    $$
        \lim_{n \to \infty}\Pp(Y \in Y_d(n, \textstyle \frac{c}{n}) \st A_Y \mbox{ fails the WLP}) = 1.
    $$
\end{theorem}

\begin{proof}
    Since $d > 0$, we may apply~\cref{t:gorensteinfailure} to conclude that for every $n$
    \begin{equation}\label{eq:prob1}
        \Pp(Y \in Y_d(n, \textstyle \frac{c}{n}) \st A_Y \mbox{ fails the WLP}) \geq \Pp(Y \in Y_d(n, \frac{c}{n}) \st f_{d - 1} \geq f_d \qand \tilde H_d(Y; \K) \neq 0).        
    \end{equation}
    Let $E_1$ denote the event $\{Y \in Y_d(n, \frac{c}{n}) \st f_{d-1} \geq f_d\}$ and $E_2$ the event $\{Y \in Y_d(n, \frac{c}{n}) \st \tilde H_d(Y; \K) \neq 0\}$.
    \cref{t:probtop} and~\eqref{eq:prob1} imply for some fixed field $\K$
    \begin{align}\label{eq:prob2}
        \lim_{n \to \infty}\Pp(Y \in Y_d(n, \textstyle \frac{c}{n}) \st A_Y \mbox{ fails the WLP}) & \geq \lim_{n \to \infty}\Pp(E_1 \cap E_2) \\
        & = \lim_{n \to \infty}\Pp(E_1) + \lim_{n \to \infty}\Pp(E_2) - \lim_{n \to \infty}\Pp(E_1 \cup E_2) \\
        & =  \lim_{n \to \infty}\Pp(E_1) + 1 - 1\\
        & = \lim_{n \to \infty}\textstyle \Pp(Y \in Y_d(n, \frac{c}{n}) \st f_{d-1} \geq f_d).
    \end{align}
    Now note that $f_d$ is a random variable given by $f_d = \sum_{i = 0}^{\binom{n}{d+1}} X_i$, where each $X_i$ is an independent Bernoulli random variable with parameter $p = \frac{c}{n}$. In particular, $f_d$ has a binomial distribution with expected value $\binom{n}{d + 1} \frac{c}{n}$. This description of $f_d$ implies 
    $$
        \textstyle\Pp(Y \in Y_d(n, p) \st f_d \geq f_{d-1}) = \Pp(f_d \geq f_{d-1}) = \Pp(f_d \geq \binom{n}{d}) = \Pp(f_d \geq (1 + \delta)\binom{n}{d + 1}\frac{c}{n}),
    $$
    where $\delta = \frac{d + 1}{c} \frac{n}{n - d} - 1 > 0$. A standard application of Chernoff's bound (see~\cite[Theorem 4.4]{M2018}) to the binomial distribution then gives us 
    \begin{align}\label{eq:prob3}
        \Pp(f_d \geq f_{d - 1}) & = \textstyle\Pp(f_d \geq (1 + \delta)\binom{n}{d + 1}\frac{c}{n}) \\
        & \leq e^{-\frac{\min(\delta, \delta^2)}{4} \frac{c}{n} \binom{n}{d + 1}}.        
    \end{align}
    Since $d, \delta > 0$ and $\binom{n}{d + 1}$ is a polynomial in $n$ of degree $d + 1$, a calculus argument implies 
    $$ 
        1 = 1 - \lim_{n \to \infty}e^{-\frac{\min(\delta, \delta^2)}{4}\frac{c}{n}\binom{n}{d + 1}} \leq \lim_{n \to \infty} \Pp(f_{d} < f_{d-1}) \leq \lim_{n \to \infty} \Pp(f_d \leq f_{d - 1}) \leq 1.
    $$
\end{proof}
 
\begin{remark}[{\bf The interval $(c_d, d + 1)$ for different values of $d$}]
    Computing approximate values of $c_d$ for different values of $d$ we have 
    $$
    \begin{array}{c|cccccccccccc} 
        d & 2 & 3 & 4 & 5 & 6 & 7 & 8  \\
        \hline
        c_d &2.783&3.91 &4.962&5.984&6.993&7.997&8.998 \\
        \end{array}
    $$
    In particular, as is mentioned in~\cite{ALLTM2013}, $\lim_{d \to \infty} c_d = d + 1$, and $c_d < d + 1$ for every $d$. In other words, even though the interval $(c_d, d + 1)$ is not empty for any $d > 0$, the length of this interval quickly approaches $0$ as $d$ gets larger. 
\end{remark}

\section{Concluding remarks and future work}\label{s:future}

One of the motivations from~\cite{L1996} to study higher stresses via inverse systems, were the known implications of the SLP of an algebra of the form $A = \frac{R}{I_\Delta + (\theta_1,\dots, \theta_{d+1})}$, to the $g$-conjecture, where $\Delta$ is a homology sphere and $\theta_1, \dots, \theta_{d+1}$ is a lsop of $I_\Delta$. Proving that an algebra $A$ satisfies the SLP can have consequences to other areas such as representation theory, and can point to interesting underlying geometric structures. A special case that has been extensive studied from both perspectives, is the coinvariant ring of the symmetric group, $\Kco(\Delta)$, where $\Delta$ is the boundary of the simplex. 

From a representation theoretic perspective, there is a special $\mathfrak{sl}_2$ action on $\Kco(\Delta)$ using a strong Lefschetz element (see for example~\cite[Theorem 3.32]{lefschetzbook} and~\cite{S80}), and from a geometric point of view, $\Kco(\Delta)$ is the cohomology ring of the flag variety (see for example~\cite[Exercise 4, Section 4.D]{hatcher} for the integral case). In view of these classical results, we ask the following, which can be seen as either equivalent, or weaker than the mentioned results (see~\cite[Corollary 1. (2)]{L1996} for the analogue question in the linear setting, which is now known).

\begin{question}[{\bf Coinvariant algebraic $g$-conjecture}]\label{q:slpco}
Let $\Delta$ be a $\K$-homology sphere. Does the ring $\Kco(\Delta)$ satisfy the SLP? In other words, for a general linear form $\ell$, does the following equality hold?
$$ 
    \dim \Sco_k(\Delta, \ell) = \max(\dim \Kco(\Delta)_k - \dim \Kco(\Delta)_{k - 1}, 0)
$$
\end{question}

We note that from a purely enumerative point of view,~\cref{q:slpco} does not imply any new interesting inequalities on the $h$-vectors of complexes where the $g$-theorem is known.

\begin{proposition}
    Let $h_i^{\co} = \dim \Kco(\Delta)_i$, where $\Delta$ is a $d$-dimensional simplicial sphere, $t = \binom{d + 2}{2}$ and set $g^{\co} = (g_0^{\co}, \dots, g_{\lceil\frac{t}{2}\rceil}^{\co})$, where $g_0^{\co} = 1$ and $g_i^{\co} = h_i^{\co} - h_{i-1}^{\co}$ for $i > 0$. Then there exists an artinian monomial ideal $J \subset R$ such that 
    $$
        \HS_\frac{R}{J}(q) = g_0^{\co} + \dots + g_{\lceil \frac{t}{2}\rceil}^{\co} q^{\lceil \frac{t}{2} \rceil}.
    $$
    In other words, $(h_0^{\co}, \dots, h_t^{\co})$ is the $h$-vector of a $(t-1)$--dimensional simplicial sphere.
\end{proposition}

\begin{proof}
    We only need to prove that there exists an algebra with the same Hilbert series as $\Kco(\Delta)$ satisfying the SLP.
    
    Since it is known that there exists an artinian reduction of $\frac{R}{I_\Delta}$ by a lsop satisfying the SLP (see~\cite{A2018,PP2020,APP2021}), let $A$ be such an algebra and $\ell_1$ a Lefschetz element of $A$. Let $B = \frac{\K[y_1, \dots, y_d]}{(e_1, \dots, e_d)}$, where $(e_1, \dots, e_d)$ is the ideal generated by the elementary symmetric polynomials, $\K$ the same base field and $\ell_2$ a Lefschetz element for $B$. Then $\ell_1 \otimes_\K 1 + 1 \otimes_\K \ell_2$ is a Lefschetz element for $C = A \otimes_\K B$~\cite[Theorem 3.34]{lefschetzbook}, and the Hilbert series of $C$ is equal to the Hilbert series of $\Kco(\Delta)$.
\end{proof}

In this paper, we used the connections between inverse systems and Lefschetz properties of monomial ideals to give a vast generalization of one of the main results from~\cite{MMN2011}, which says a special class of monomial almost complete intersections fails the WLP. Since~\cite[Conjecture 6.8]{MMN2011}, there have been many interesting connections developed between Lefschetz properties of monomial ideals, and enumerations of combinatorial objects. A special example is the work of Cook II and Nagel~\cite{CN2011}, where they use lozenge tilings to study Lefschetz properties of ideals of the form $(x^{\alpha}, y^{\beta}, z^{\gamma}, x^a y^b z^c)$. 

The interesting combinatorial objects that characterize Lefschetz properties of ideals such as monomial (almost) complete intersections (in $3$ variables), together with the often different behaviour for higher codimension, leads us to the following question.

\begin{question}[{\bf Monomial artinian reductions of Gorenstein squarefree monomial ideals}]
    Let $\Delta$ be a $d$-dimensional (Cohen-Macaulay) orientable pseudomanifold over a field $\K$, and $I_\Delta(a_1, \dots, a_n) = I_\Delta + (x_1^{a_1}, \dots, x_n^{a_n})$. Let 
    $$
        X_\Delta = \{(a_1, \dots, a_n) \in \N^n \st \frac{\K[x_1, \dots, x_n]}{I_\Delta(a_1, \dots, a_n)} \mbox{ has the WLP}\} \subseteq \N^n.
    $$
    Assuming $d \gg 0$, is the set $X_\Delta$ bounded? Is it enough to take $d \geq 4$?
\end{question}

Finally, there are finer invariants that one may consider instead of the WLP. Computational evidence suggests the following.

\begin{question}[{\bf The generic Jordan type of homology-spheres}]
    Let $\Delta$ be a $d$-dimensional $\Q$-homology sphere and $J = I_\Delta + (x_1^{d+2}, \dots, x_n^{d+2})$. Then is there a unique multiplication map by $L = x_1 + \dots + x_n$ that fails to have full rank? Does this map only fail by $1$? More generally, what can be said about the ranks of the maps $\times L^i : \frac{R}{J} \to \frac{R}{J}$? 
\end{question}

\acknowledgement{I would like to thank Martin Winter for insightful conversations on rigidity theory, Margaret Bayer for insightful conversations on $f$-vectors, and
my advisor Sara Faridi for comments that improved the exposition of this paper.}

\bibliographystyle{plain}
\bibliography{bibliography.bib}

\end{document}